\documentclass[11pt]{amsart}

\usepackage[utf8]{inputenc}
\usepackage[T1]{fontenc}
\usepackage[french, english]{babel}

\usepackage{amsmath}
\usepackage{amssymb}
\usepackage{mathrsfs}
\usepackage{bbm}
\usepackage{graphicx}
\usepackage{listings}
\usepackage{mathrsfs}
\usepackage{enumerate}
\usepackage{pict2e}
\usepackage{stmaryrd}
\usepackage{mathtools}

\usepackage{enumitem}

\usepackage[all,cmtip]{xy}

\usepackage{multirow}

\usepackage{color}
\definecolor{marin}{rgb}   {0.,   0.1,   0.5} 
\definecolor{rouge}{rgb}   {0.8,   0.,   0.} 
\definecolor{sepia}{rgb}   {0.4,   0.25,   0.} 
\definecolor{mag}{rgb}   {0.3,   0,   0.3} 

\usepackage[colorlinks,citecolor=sepia,linkcolor=marin,urlcolor=mag ,bookmarksopen,bookmarksnumbered]{hyperref}

\newtheorem{theorem}{Theorem}[section]
\newtheorem{corollary}[theorem]{Corollary}
\newtheorem{lemma}[theorem]{Lemma}
\newtheorem{proposition}[theorem]{Proposition}

\newtheorem{remark}[theorem]{Remark}

\DeclareMathOperator{\Reelle}{Re}

\DeclareMathOperator{\Imag}{Im}

\newcommand*{\transp}[2][-3mu]{\ensuremath{\mskip1mu\prescript{\smash{\mathrm t\mkern#1}}{}{\mathstrut#2}}}                      

\begin{document}

\title[$L^p$ estimates and local smoothing effects for the quadratic evolution equations]{Gains of integrability and local smoothing effects for quadratic evolution equations}

\author{Paul Alphonse}
 \address{\small{(Paul Alphonse) Universit\'e de Lyon, ENSL, UMPA - UMR 5669, F-69364 Lyon}}

\email{paul.alphonse@ens-lyon.fr}

\author{Joackim Bernier}

 \address{\small{(Joackim Bernier) Laboratoire de Math\'ematiques Jean Leray, Universit\'e de Nantes, UMR CNRS 6629\\
2 rue de la Houssini\`ere \\
44322 Nantes Cedex 03, France}}

\email{joackim.bernier@univ-nantes.fr}


\makeatletter
	\@namedef{subjclassname@2020}{\textup{2020} Mathematics Subject Classification}
\makeatother


\begin{abstract} 
We characterize geometrically the semigroups generated by non-selfadjoint quadratic differential operators $(e^{-tq^w})_{t\geq 0}$ enjoying local smoothing effects and providing gains of integrability. More precisely, we prove that the evolution operators $e^{-tq^w}$ map $L^{\mathfrak{p}}$ on $L^{\mathfrak{q}} \cap C^\infty$, for all $1\leq \mathfrak{p} \leq \mathfrak{q} \leq \infty$, if and only if the singular space of the quadratic operator $q^w$ is included in the graph of a linear map. We also provide quantitative estimates for the associated operator norms in the short-time asymptotics $0<t\ll 1$.
\end{abstract} 
\maketitle


\section{Introduction}\label{intro}

\subsection{Motivation}
During the last decades, remarkable advances have been made in the analysis of the smoothing effects of semigroups generated by non-self-adjoint operators \cite{Wei79,Her07,HP09,HPV17,HPV18,Alp20a,AB20,AB21,Whi21a,Whi21b}. One of the main objectives of this line of research is to understand (and to quantify) how the interactions between regularizing phenomena (encoded by the self-adjoint part of the operator) and transport dynamics (encoded by its skew-adjoint part) enhance the regularizing effects of the semigroup. For instance, if we consider the Kolmogorov equation 
$$
\partial_t u = v\partial_x  u + \partial_v^2 u, \quad t\geq 0,\, x,v\in \mathbb{R},
$$
at first glance, we expect that its solutions are very smooth with respect to the speed variable $v$ but not especially with respect to the space variable $x$. Nevertheless, the Kolmogorov's splitting formula (see \cite{Kol34})
\begin{equation}
\label{eq:split_kolmo}
e^{t (v\partial_x + \partial_v^2)} = e^{t (\partial_v - t\partial_x/2)^2 + t^3 \partial_x^2/12 }e^{tv\partial_x},
\end{equation}
makes clear that the solutions are actually also smooth with respect to the space variable $x$ (but the smoothing effects are slower in this direction due to the factor $t^3$). Moreover, it can be noted on \eqref{eq:split_kolmo} that this gain is due to the non-commutation (i.e. the interaction) between the free transport $v\partial_x$ and the partial diffusion $\partial_v^2$.

In this paper, we focus on semigroups generated by quadratic differential operators acting on $L^2(\mathbb R^n)$, with $n\geq1$. They are the evolution operators associated with partial differential equations of the form
\begin{equation}\label{eq:Navier-Stokes}
\partial_t u + q^w(x,D_x) u = 0,\quad t\geq0,\, x\in\mathbb R^n,
\end{equation}
where $q^w(x,D_x)\equiv q^w$ is the Weyl quantization of a complex-valued quadratic form $q:\mathbb R^{2n}\rightarrow\mathbb C$ with a nonnegative real part, and $D_x = -i\nabla$. Denoting $Q\in S_{2n}(\mathbb C)$ the matrix of $q$ in the canonical basis of $\mathbb R^{2n}$, that is, the unique $Q\in S_{2n}(\mathbb C)$ such that $q(X) = QX\cdot X$ for all $X\in\mathbb R^{2n}$, $q^w$ is nothing but the differential operator 
$$q^w(x,D_x) = \begin{pmatrix} x & D_x \end{pmatrix} Q \begin{pmatrix} x \\ D_x
\end{pmatrix}.$$
We point out that since the quadratic form $q$ is complex-valued, the quadratic operator $q^w$ is generally non-selfadjoint: $(\Reelle q)^w$ is selfadjoint whereas $i(\Imag q)^w$ is skew-selfadjoint. This operator is equipped with the domain $D(q^w) = \{u\in L^2(\mathbb R^n) : q^wu\in L^2(\mathbb R^n)\}.$ 
Let us recall that since the real part of the quadratic form $q$ is nonnegative, the quadratic operator $q^w$ is shown in  \cite[pp. 425-426]{Hor95} to be maximal accretive and to generate a strongly continuous contraction semigroup $(e^{-t q^w})_{t\geq0}$ on $L^2(\mathbb R^n)$. This setting is not very general but it  contains a lot of non trivial examples (i.e. for which the regularizing effects are not directly given by an explicit formula like \eqref{eq:split_kolmo}). Moreover, this framework includes several interesting equations coming from physics (like Fokker-Planck equations for instance, see \cite[Section 4]{Ber21}).


Being given a complex-valued quadratic form $q:\mathbb R^{2n}\rightarrow\mathbb C$, its \emph{singular space} $S$ is the subspace of $\mathbb{R}^{2n}$ defined by\footnote{Note that sometimes in the literature, $\Reelle Q$ is replaced by $\Reelle F$ in the definition of $S$. Nevertheless, since the matrix $J$ is real and invertible, the definitions are of course equivalent.}
\begin{equation}\label{def:S} 
	S := \bigcap_{\ell \in \mathbb{N}} \mathrm{Ker}\big((\Reelle Q) (\Imag F)^\ell\big),
\end{equation}
where $F = JQ$ is the \emph{Hamilton map} of $q$ and $J$ is the matrix of the standard symplectic form, i.e.
$$
J = \begin{pmatrix} 0_n & I_n \\
-I_n & 0_n
\end{pmatrix}.
$$
This space, which has been introduced for the first time by Hitrik and Pravda-Starov in \cite{HP09}, encodes most of the smoothing effects of the semigroup $(e^{-t q^w})_{t \geq 0}$. For example, we have proven in \cite[Thm 2.6 and 2.8]{AB21} that $S^{\perp}$ (the orthogonal complement of $S$ for the canonical Euclidean structure of $\mathbb{R}^{2n}$) is the set of the directions of the phase space along which $e^{- q^w}$
is regularizing\footnote{since $q$ and $tq$ have the same singular space, the time plays no role in this discussion, so we choose $t=1$.}, i.e.
\begin{equation}
\label{eq:weakAENS}
	S^{\perp} = \big\{ (x_0,\xi_0) \in \mathbb R^{2n} \ |\ (x_0 \cdot x + \xi_0 \cdot D_x)\, e^{- q^w} \ \mathrm{is \ bounded \ on} \ L^2(\mathbb{R}^n) \big\}.
\end{equation}
Actually, in \cite[Thm 2.6]{AB21}, we have proven a result much stronger than \eqref{eq:weakAENS}: if $e^{- q^w}$ is multiplied by any product of operators of the form $(x_0 \cdot x + \xi_0 \cdot D_x)$, with $(x_0,\xi_0)\in S^{\perp}$, then it is still bounded on $L^2(\mathbb{R}^n)$. As a consequence, focusing only on the smoothing effects of $e^{- q^w}$, we have proven that
\begin{equation}
\label{eq:joli}
	S \subset \mathbb{R}^n  \times \{ 0 \} \quad \iff \quad e^{- q^w} \ \mathrm{maps} \ L^2(\mathbb{R}^n) \ \mathrm{to} \ H^\infty(\mathbb{R}^n).
\end{equation}
Nevertheless, in \eqref{eq:joli}, we miss some local smoothing effects. Indeed, consider for example the self-adjoint quadratic operator associated with the quadratic form $q(x,\xi)=(\xi - x)^2$ on $\mathbb{R}^2$. Its singular space is the diagonal $S=\{ (x,x) \ | \ x\in \mathbb{R} \}$, so \eqref{eq:joli} does not apply. However, since 
$$
e^{ \frac{i}2 x^2} \partial_x^2 e^{- \frac{i}2 x^2} = - (D_x - x)^2,
$$
we have
$$
e^{- (D_x - x)^2} = e^{ \frac{i}2 x^2} e^{\partial_x^2} e^{- \frac{i}2 x^2}.
$$
Therefore, we note that this semigroup makes the functions smooth (but not uniformly in space due to the highly oscillatory factor $e^{- \frac{i}2 x^2}$).

\subsection{Main results}
One of the main achievements of this paper is to take into account this kind of local smoothing phenomena: we characterize geometrically the locally smoothing semigroups. Indeed, the following theorem is a direct qualitative corollary of our results (i.e. Theorem \ref{thm:maindir} and Theorem \ref{thm:recip} just below).
 \begin{theorem} \label{thm:weak} For any complex-valued quadratic form $q$ on $\mathbb{R}^{2n}$ whose real part is nonnegative, denoting by $S$ its singular space (defined by \eqref{def:S}), we have
\begin{equation}
\label{eq:tresjoli}
	S \cap  (\mathbb{R}^n  \times \{ 0 \})^\perp = \{0\} \quad \iff \quad e^{- q^w} \ \mathrm{maps} \ L^2(\mathbb{R}^n) \ \mathrm{to} \ C^\infty(\mathbb{R}^n).
\end{equation}
\end{theorem}
\begin{remark} 
\label{rem:graph}
The condition $S \cap  (\mathbb{R}^n  \times \{ 0 \})^\perp = \{0\}$ means that $S$ has to be included in a graph, i.e. there exists a $n\times n$ real matrix $G$ such that $S \subset \{ (x,Gx) \ | \ x\in \mathbb{R}^n \} $. See Lemma \ref{lem:graph} in the appendix for the proof.
\end{remark}
\begin{remark} It is relevant to compare \eqref{eq:joli} and \eqref{eq:tresjoli}. On the one hand, the space $C^\infty(\mathbb{R}^n)$ is much larger than $H^\infty(\mathbb{R}^n)$ but on the other hand, it is much less restrictive to have the condition $S \cap  (\mathbb{R}^n  \times \{ 0 \})^\perp = \{0\}$ than 
$S \subset \mathbb{R}^n  \times \{ 0 \}$. Indeed, in the vain of Remark \ref{rem:graph}, the assumption $S \subset \mathbb{R}^n  \times \{ 0 \}$ means that $S$ has to be included in the graph of the null matrix.
\end{remark}
\begin{remark} As we will see in Theorem \ref{thm:maindir} and Theorem \ref{thm:recip}, in \eqref{eq:tresjoli}, the space $L^2(\mathbb{R}^n)$ could be replaced by any space $L^\mathfrak{p}(\mathbb{R}^n)$ with $1\leq \mathfrak{p} \leq \infty$.
\end{remark}

\begin{remark}  The implication ``$\Rightarrow$'' in Theorem \ref{thm:weak} is also a direct corollary of some results in the literature on propagation of singularities, see e.g. \cite[Theorem 6.2]{PRW18}, \cite[Proposition 3]{Wah18}.
\end{remark}
%

Studying carefully the new condition $S \cap  (\mathbb{R}^n  \times \{ 0 \})^\perp = \{0\}$, we understood that it is not only associated with local smoothing properties but also with some gains of integrability. As we will see in Theorem \ref{thm:maindir} just below, the evolution operator can be extended  by continuity from $L^{\mathfrak{p}}$ to $L^{\mathfrak{q}}$ provided that $\mathfrak{p} \leq \mathfrak{q}$. Somehow, this gain of integrability (which is the same as the one for the heat equation) can be seen as another kind of regularizing effect. It seems to us that it is an interesting property in itself but we also point out that it may be valuable to study the well-posedness of nonlinear perturbations of the equation. 
 \begin{theorem} \label{thm:weakLp} For any complex-valued quadratic form $q$ on $\mathbb{R}^{2n}$ whose real part is nonnegative, denoting by $S$ its singular space (defined by \eqref{def:S}), we have 
\begin{equation}
\label{eq:tresjolilp} 
	S \cap  (\mathbb{R}^n  \times \{ 0 \})^\perp = \{0\} \ \iff \ \forall 1\leq \mathfrak{p}\leq \mathfrak{q} \leq \infty, \  e^{- q^w} \ \mathrm{is \ bounded \ from} \ L^{\mathfrak{p}}(\mathbb{R}^n) \ \mathrm{to} \ L^{\mathfrak{q}}(\mathbb{R}^n). 
\end{equation} 
\end{theorem}

This theorem is a qualitative corollary of Theorems \ref{thm:maindir} and \ref{thm:recip} below. It extends the recent result of White \cite{Whi21b} which analyses this gains of integrability in the hypoelliptic case $S=\{0\}$. As we will see at the end of this introduction (see \eqref{eq:Gaussker}), when $S \cap  (\mathbb{R}^n  \times \{ 0 \})^\perp = \{0\}$ it can be proven that $e^{- q^w} $ is actually an integral transform associated with a complex Gaussian kernel (possibly degenerated). 
The gain of integrability of this kind of kernel has been widely studied (see e.g. \cite{Wei79,Lieb90,Neg95}). Nevertheless, it seems to us that Theorem \ref{thm:weakLp} cannot be directly deduced of these results (see the discussion about this point in subsection \ref{subsec:disc}).

Now, we shall give a quantitative version of Theorems \ref{thm:weak} and \ref{thm:weakLp}. As usual, we have to introduce the notion of \emph{global index} which quantify how fast are the smoothing effects. It is denoted by $k_0$ and is defined by
\begin{equation}\label{def:k_0} 
k_0 := \min \bigg\{ k \in \mathbb{N} \ | \ S = \bigcap_{\ell =0}^{k} \mathrm{Ker}(\Reelle Q) (\Imag F)^\ell \bigg\}.
\end{equation}
Note that, as a consequence of Cayley-Hamilton's theorem, we have $0\le k_0 \leq 2n-1$. The following theorem is the main result of this paper. It provides a quantitative estimate of the regularizing effects.

\begin{theorem} \label{thm:maindir} Let $q$ be a complex-valued quadratic form on $\mathbb{R}^{2n}$ whose real part is nonnegative, $S$ be its singular space (defined by \eqref{def:S}) and $k_0$ be its global index (defined by \eqref{def:k_0}).

If $S$ is included in the graph of a real $n \times n$ matrix $G$, i.e. $S \subset \{ (x,Gx) \ | \ x\in \mathbb{R}^n \}$, then for all $1 \leq \mathfrak{p} \leq \mathfrak{q} \leq +\infty$ and all $t>0$, $e^{-tq^w}$ can be extended by continuity in such a way that
$$e^{- tq^w} \ \ \mathrm{maps} \ \ L^{\mathfrak{p}}(\mathbb{R}^n) \ \ \mathrm{to} \ \ L^{\mathfrak{q}}(\mathbb{R}^n) \cap C^\infty(\mathbb{R}^n),$$
and, for all $ u \in L^{\mathfrak{p}}(\mathbb{R}^n)$ and $m \geq 0$, provided that $0<t\leq 1$, we have the quantitative estimate
\begin{equation}
\label{eq:mainest}
  \big\| (\langle G x \rangle + \langle \transp{G} x \rangle)^{-m} \mathrm{d}^m (e^{-tq^w} u) \big\|_{L^{\mathfrak{q}}} \leq  \frac{C^{1+m}}{t^{(k_0+\frac12)m + c_{\mathfrak p,\mathfrak q}}}\ \sqrt{m!}\ \|u\|_{L^{\mathfrak p}},
\end{equation}
where $\mathrm{d}$ denotes the Fr\'echet derivative and $C>0$ is a constant which is uniform with respect $t,\mathfrak{p},\mathfrak{q},u,m$ and
\begin{equation} \label{21102021E7} c_{\mathfrak p,\mathfrak q} = \left\{\begin{array}{ll}
	\frac n{2\mathfrak r}(2k_0 + \mathfrak r - 1) & \text{when $1\le\mathfrak r\le 2$,} \\[5pt]
	\frac n{2\mathfrak r}(2k_0+1)(\mathfrak r-1) & \text{when $\mathfrak r>2$,}
\end{array}\right. \end{equation}
with
$$\mathfrak r = (1- \mathfrak p^{-1} + \mathfrak q^{-1} )^{-1} \in [1,+\infty].$$
\end{theorem}

\begin{remark} Due to the factor $(\langle G x \rangle + \langle \transp{G} x \rangle)^{-m}$ in \eqref{eq:mainest}, the smoothing effects are only local. However, when $S \subset \mathbb{R}^n  \times \{ 0 \}$ (as in \eqref{eq:joli}), we can choose $G=0$ and so \eqref{eq:mainest} proves that the smoothing effects are global.
\end{remark}
\begin{remark} In \cite[Thm 1.1]{Whi21b}, White proves that whenever $S=\{0\}$, $e^{- tq^w}$ is bounded from $L^{\mathfrak{p}}(\mathbb{R}^n)$ to $L^{\mathfrak{q}}(\mathbb{R}^n)$ (with $1 \leq \mathfrak{p} \leq \mathfrak{q} \leq +\infty$) and enjoys the bound $\| e^{- tq^w}\|_{L^{\mathfrak{p}} \to L^{\mathfrak{q}}} \lesssim t^{-(2k_0+1)n} $. Observing that in any case, $c_{\mathfrak{p},\mathfrak{q}} \leq \frac12 (2k_0+1)n$, the bound given by Theorem \ref{thm:maindir} improves the exponent given in \cite{Whi21b} of a factor at least equal to $1/2$.
\end{remark}

Finally, we provide a version of the reciprocal of Theorems \ref{thm:weak} and \ref{thm:weakLp}. 

\begin{theorem} \label{thm:recip}For any complex-valued quadratic form $q$ on $\mathbb{R}^{2n}$ whose real part is nonnegative, denoting by $S$ its singular space (defined by \eqref{def:S}), if
$$
e^{- q^w} \ \mathrm{maps} \ L^{2}(\mathbb{R}^n) \ \mathrm{to} \ C^0(\mathbb{R}^n) \quad \mathrm{or} \quad   \exists \mathfrak{q}>2, \quad  e^{- q^w} \ \mathrm{maps} \ L^{2}(\mathbb{R}^n) \ \mathrm{to} \ L^{\mathfrak{q}}(\mathbb{R}^n),
$$
then
$
  S \cap  (\mathbb{R}^n  \times \{ 0 \})^\perp = \{0\}.
$
\end{theorem}

\subsection{Heuristic} Now, we discuss the heuristic behind our results and our strategy of proof.

First, let us focus on the reciprocals, i.e. Theorem \ref{thm:recip}. In \cite[Thm 2.1]{AB21}, we have proven that the polar decomposition of $e^{- q^w}$ is of the form
\begin{equation}
\label{eq:ours}
e^{- q^w} = e^{-a^w} U, 
\end{equation}
where $U$ is unitary on $L^2(\mathbb{R}^n)$ and $a$ is a real-valued nonnegative quadratic form (i.e. $a^w$ is self-adjoint) about which we have proven a lot of properties. In particular, using the tools introduced in \cite{AB21}, it can be proven quite easily that $S$ is the isotropic cone\footnote{Recall that the isotropic cone of a quadratic form $q:\mathbb R^{2n}\rightarrow\mathbb C$ is defined by $C(q) = \big\{X\in\mathbb R^{2n} : q(X) = 0\big\}$.} of $a$ (see Proposition \ref{prop:singspacepolar} in the appendix).  If $S \cap  (\mathbb{R}^n  \times \{ 0 \})^\perp \neq \{0\}$, up to a change of variable (in the physical space $\mathbb{R}^n$), we can assume without loss of generality that $(0,e_n) \in S$ (where $(e_j)_j$ is the canonical basis of $\mathbb{R}^n$). Since $a$ vanishes on $S$ and $a$ is non-negative, this means that $a$ does not depend on $\xi_n$ or in other words, that $\partial_{x_n}$ does not appear in $a^w$. Coming back to the polar decomposition \eqref{eq:ours} of $e^{-q^w}$, it is clear that since $U$ is unitary, it is not regularizing. Moreover, since $\partial_{x_n}$ does not appear in $a^w$, it is also quite intuitive that $e^{-a^w}$ (and so $e^{-q^w}$) cannot provide any regularizing effect with respect to the variable $x_n$. The rigorous proofs of these reciprocals are presented in Section \ref{sec:reciprocals}.

\medskip

Now, we focus on the heuristic of the proof of the regularizing effects, i.e. Theorem \ref{thm:maindir}. The most difficult part consists in proving that, under the geometrical assumption on the singular space, $e^{- q^w}$ maps $L^{\mathfrak{p}}(\mathbb{R}^n)$ on $L^{\mathfrak{q}}(\mathbb{R}^n)$ when $1 \leq \mathfrak{p} \leq \mathfrak{q} \leq +\infty$. Indeed, the operator $U$ in  \eqref{eq:ours} can generate a lot of trouble. For instance, $U$ could be the evolution operator of the Schr\"odinger equation, i.e. $U=e^{i \Delta}$, and so it would not be well-defined on $L^{\mathfrak{p}}$ when $\mathfrak{p} > 2$ and would map $L^{\mathfrak{p}}$ on $L^{\mathfrak{p}'}$  when $\mathfrak{p} \leq 2$ (where $(\mathfrak{p}')^{-1} + (\mathfrak{p})^{-1} = 1$). We could even imagine worst situations where $U$ would be the evolution operator of a degenerated\footnote{i.e. $U=e^{ i D\nabla \cdot \nabla}$ where $D$ is symmetric but not not invertible.} Schr\"odinger equation. Therefore, we have to prove that $e^{-a^w}$ is able to compensate such bad behaviors. Of course, the best situation would be to be able to decompose $e^{-a^w}$ under the form $B e^{- P \nabla \cdot \nabla}$, with $P$ a positive $n\times n$ symmetric real matrix and $B$ a bounded operator on $L^{\mathfrak{q}}$. Indeed, an operator of the form $ e^{- K \nabla \cdot \nabla}$, where $K=P+iD$ is a complex symmetric matrix whose real part is positive, has the same regularizing effects as the evolution operator of the heat equation. Unfortunately, we did not succeed to prove that such a decomposition holds, and we conjecture that in general, it does not exist. Therefore, in Section \ref{sec:dec}, we prove a new kind of decomposition of $e^{-q^w}$ which is finer than \eqref{eq:ours} but holds only under the geometrical assumption $S \cap  (\mathbb{R}^n  \times \{ 0 \})^\perp = \{0\}$. More precisely, if $G$ a real $n\times n$ matrix whose graph contains $S$, then, provided that $t$ is small enough\footnote{here it is relevant to take into account the parameter $t$ in order to have quantitative estimates. In order to get the regularizing effects of $e^{-tq^w}$, it is sufficient to use the semigroup property.}, we prove in Theorem \ref{thm:the_dec} that the following decomposition holds
\begin{equation}
\label{eq:madecintro}
e^{-tq^w} = c_t \, e^{ \frac{i}2 G x \cdot x} (e^{-\gamma t^\alpha |\xi -  N x|^2 })^w e^{-t p_t^w} (e^{-\gamma t^{\alpha} |\xi - Nx|^2 })^w e^{i t D_t \nabla \cdot \nabla} e^{tM_t x \cdot \nabla}  e^{\frac{i}2 (tW_t-  G ) x \cdot x},
\end{equation}
where $\gamma>0$ is a constant that does not depend on $t$, $\alpha = 2k_0+1$, $N = (G-\transp{G})/2$ denotes the skew-symmetric part of $G$, $D_t,M_t,W_t$ are some real $n\times n$ matrices, $c_t$ is a real constant and $p_t$ is a real-valued nonnegative quadratic form, all depending smoothly on $t$. The pseudo-differential operator $(e^{-\gamma t^{\alpha} |\xi - Nx|^2 })^w $ is defined as usual through an oscillatory integral (see \eqref{eq:defWeyl}). 
The existence of such decompositions is strongly related to the exact classical-quantum correspondence (see Prop \ref{prop:split}). Thanks to the decomposition \eqref{eq:madecintro}, it is much easier to understand the action of $e^{-tq^w}$ on $L^{\mathfrak{p}}$. Indeed, we just have to analyse the action of each factor (this is done in Section \ref{sec:Gauss}):
\begin{enumerate}[label=\textbf.,leftmargin=* ,parsep=2pt,itemsep=0pt,topsep=2pt]
	\item[$\cdot$] $e^{\frac i2 (tW_t-  G ) x \cdot x}$ is an isometry on $L^{\mathfrak p}$, so it plays no role.
	\item[$\cdot$] $e^{tM_t x \cdot \nabla}$ is the flow of a transport equation, so it is invertible on $L^{\mathfrak p}$ and is close to the identity, so it plays no role either.
	\item[$\cdot$] As explained previously, $e^{i t D_t \nabla \cdot \nabla}$ is the most dangerous term. We have to take into account the action of the next term to analyse it.
	\item[$\cdot$] In Section \ref{sec:Gauss}, we establish several results proving that, somehow, $(e^{-\gamma t^{\alpha} |\xi - Nx|^2 })^w$ behaves like $e^{\gamma t^{\alpha} \Delta }$. More precisely, we prove in Corollary \ref{cor:Ng} that this operator is nothing but the following integral transform with Gaussian kernel
\begin{equation}\label{eq:introreftwisted}
	(e^{-\gamma t^{\alpha}|\xi - Nx|^2 })^w u(x) =(4\pi \gamma t^{\alpha} )^{-n/2}  \int_{\mathbb R^n} e^{-\frac1{4\gamma t^{\alpha}} |x-y|^2 + i (x-y) \cdot N x }u(y)  \, \mathrm{d}y. 
\end{equation}
In particular, it allows us to prove in Corollary \ref{cor:twisted_vs_schro} that $(e^{-\gamma t^{\alpha} |\xi - Nx|^2 })^w e^{i t D_t \nabla \cdot \nabla}$ maps $L^{\mathfrak{p}}$ to $L^{\mathfrak{q}}$.
	\item[$\cdot$] Since $p_t$ is real-valued, thanks to the Mehler formula, we prove in Theorem \ref{thm:fifou} that $e^{-t p_t^w}$ is bounded on $L^{\mathfrak{q}}$. So it plays no role.
	\item[$\cdot$] In Section \ref{sec:localsmooth}, we prove that  $e^{ \frac{i}2 G x \cdot x} (e^{-\gamma t^\alpha |\xi -  N x|^2 })^w$ is locally smoothing (i.e. that it maps $L^{\mathfrak{q}}$ to $L^{\mathfrak{q}} \cap C^\infty$) which is quite natural if we keep in mind that $(e^{-\gamma t^{\alpha} |\xi - Nx|^2 })^w$ behaves like $e^{\gamma t^{\alpha} \Delta }$ (or considering directly the formula \eqref{eq:introreftwisted}).
\end{enumerate}

\subsection{Further discussions about an alternative approach} \label{subsec:disc}

Here, in the spirit of our previous works \cite{AB20,Ber21,AB21}, we have chosen to study the regularizing effects by decomposing $e^{-tq^w}$ as a product of evolution operators which are simple to analyse, but it is not the only possible approach. In particular, now, we would like to mention and discuss the difficulties related to another one which seems however very natural. Indeed, thanks to the Mehler formula (see Theorem \ref{thm:Meh}) established by H\"ormander in \cite{Hor95}, $e^{-q^w}$ is actually a pseudo-differential operator whose symbol is given\footnote{up to an algebraic condition.} by $c e^{- m}$ where $c\in\mathbb C$ is a constant and $m$ is a complex valued quadratic form on $\mathbb R^{2n}$. Moreover, $\Reelle m$ is nonnegative and its isotropic cone is the singular space $S$ (see \cite[Lemma 3.1]{Alp20a}). Therefore, decomposing $m$ as
$$
m(x,\xi) =  \frac12 r(x) + L x \cdot \xi + \frac12 b(\xi),
$$
the geometric condition $S \cap  (\mathbb{R}^n  \times \{ 0 \})^\perp = \{0\}$ implies that $\Reelle b$ is positive. Consequently, using the definition of $(e^{-m})^w$ as an oscillatory integral (see \eqref{eq:defWeyl}), long and tedious computations\footnote{that we do not detail in this paper, due to their length. However, they are somehow similar to the ones in the proof of Corollary \ref{cor:g}.} allow to prove that $e^{-q^w}$ is nothing but an integral transform with Gaussian kernel of the form
\begin{equation}
\label{eq:Gaussker}
e^{-q^w}u(x) = \widetilde{c} \int_{\mathbb{R}^n} e^{-\frac12 k(x,y) }u(y) \, \mathrm{d}y,
\end{equation}
where $\widetilde{c} \in \mathbb{C}$ is a complex constant, and $k$ is a complex-valued quadratic form whose real part is nonnegative, and takes the form
$$
\Reelle k(x,y) = v\Big(\frac{x+y}2\Big) + p( Mx - Ny). 
$$
Here, $p$ (resp. $v$) is a positive (resp. nonnegative\footnote{Indeed since $m$ is nonnegative it is a consequence of \eqref{eq:canon}.}) real-valued quadratic form whose matrix, denoted $P$ (resp. $V$), is given by
$$
P  =  ( \Reelle B + (\Imag B) (\Reelle B)^{-1} (\Imag B) )^{-1}   \quad \mathrm{and} \quad V = \Reelle R - (\Reelle L) (\Reelle B)^{-1} (\Reelle L),
$$
where $R$ and $B$ denote respectively the matrices of $r$ and $b$, and $M,N$ are the real $n\times n$ matrices given by
$$
M = I_n - \frac12 \Imag L + \frac12 (\Imag B) (\Reelle B)^{-1} (\Reelle L)  \quad \mathrm{and} \quad N=I_n + \frac12 \Imag L - \frac12 (\Imag B) (\Reelle B)^{-1} (\Reelle L).
$$
Since we know that $\Reelle b$ is positive, $\Reelle  m$ is nonnegative and $S$ is the isotropic cone of $\Reelle  m$, it follows that
\begin{equation}
\label{eq:canon}
2 \Reelle  m(x,\xi ) = v(x) + \Reelle b(\xi + (\Reelle B)^{-1}(\Reelle L) x),
\end{equation}
then $v$ is typically degenerated and the dimension of its kernel is
$$
\mathrm{dim} \, \mathrm{Ker} \, V = \mathrm{codim} \, S.
$$
Therefore, $v$ could vanish identically\footnote{i.e. we could have $\mathrm{dim} \, S = n$ (which is the maximum with respect to the assumption $S \cap  (\mathbb{R}^n  \times \{ 0 \})^\perp = \{0\}$).} and so we should not expect any regularizing effect coming from $v$. Conversely, if both the matrices $M$ and $N$ are invertible, thanks to the integral representation  \eqref{eq:Gaussker}, it is obvious\footnote{the proof is the same as for the heat equation.} that $e^{-q^w}$ is locally smoothing and maps $L^{\mathfrak{p}}$ in $L^{\mathfrak{q}}$ when $1\leq \mathfrak{p} \leq \mathfrak{q} \leq \infty$. However, else if $M$ or $N$ is non-invertible, it is not sufficient to consider the real part of $k$ to study the regularizing effects of $e^{-q^w}$, and so the situation is much more intricate: the role of the complex phase has to be taken into account. Moreover, it seems that there is another obstacle to study the regularizing effects of $e^{-q^w}$ by following this way: the quadratic form $m$ enjoys some specific algebraic properties which should be determined and taken into account. For example, the real part of the Fokker-Plank quadratic form $m(x,\xi) =\frac12 \xi^2 - 2ix \xi$ is nonnegative, its isotropic cone $S = \mathbb{R} \times \{ 0 \}$ satisfies the geometric condition $S \cap  (\mathbb{R}  \times \{ 0 \})^\perp = \{0\}$, but a direct computation shows that
$$
\forall u \in \mathscr{S}(\mathbb{R}), \quad (e^{- \frac12 \xi^2 + 2ix \xi})^w u(x) = (2\pi)^{-1/2} e^{-x^2/2} \int_{\mathbb{R}}   u(y) \, \mathrm{d}y.
$$
Therefore, $(e^{- \frac12 \xi^2 + 2ix \xi})^w$ does not enjoy the regularizing effects given by Theorem \ref{thm:weakLp}, but it is not a counter-example to our results because, since it is unbounded on $L^2$, it cannot be of the form $e^{-q^w}$ with $\Reelle q \geq 0$.

\subsection{Notations} The following notations will be used all over the work:
\begin{enumerate}[label=\textbf{\arabic*.},leftmargin=* ,parsep=2pt,itemsep=0pt,topsep=2pt]
	\item The Weyl quantization of tempered distribution $a \in \mathscr{S}'(\mathbb{R}^{2n})$ is denoted by $a^w$ and is formally defined by the following oscillatory integral
\begin{equation}\label{eq:defWeyl}
	a^wu(x) =  (2\pi)^{-n} \iint_{\mathbb R^{2n}} e^{i(x-y)\cdot\xi} a\Big(\frac{x+y}2,\xi\Big) u(y)\, \mathrm dy\, \mathrm d\xi.
\end{equation}
	\item Implicitly, $\mathbb{R}^n$ is always equipped with its natural Euclidean structure: $|\cdot|$ is the Euclidean norm while $\cdot$ is the scalar product. We also extend this scalar product by analyticity to $\mathbb C^n$. The notation $\perp$ always stands for the Euclidean orthogonal complement.
	\item The convolution product between two functions $f$ and $g$ defined on $\mathbb R^n$ is denoted $f\ast g$ and given (when it makes sense) by
	$$(f\ast g)(x) = \int_{\mathbb R^n}f(x-y)g(y)\,\mathrm dy.$$
\end{enumerate}

\section{A new decomposition} \label{sec:dec}

This section is devoted to the proof of the following theorem in which we decompose the semigroups in factors which are much simpler to study.
\begin{theorem} \label{thm:the_dec} Let $q$ be a complex-valued quadratic form on $\mathbb{R}^{2n}$ whose real part is nonnegative, $S$ be its singular space (defined by \eqref{def:S}) and $k_0$ be its global index (defined by \eqref{def:k_0}).

If $S$ is included in the graph of a real $n \times n$ matrix $G$, i.e. $S \subset \{ (x,Gx) \ | \ x\in \mathbb{R}^n \}$, then there exist two positive constants $\gamma,t_0 > 0$, some real-valued nonnegative quadratic forms $p_t$, some real $n \times n$ matrices $D_t,M_t,W_t$ and some real constants $c_t $, all depending
smoothly on $t\in(-t_0,t_0)$, such that, for all $t\in [0,t_0)$, the following decomposition holds
\begin{equation}
\label{eq:the_dec}
e^{-tq^w} = c_t \, e^{ \frac{i}2 G x \cdot x} (e^{-\gamma t^\alpha |\xi -  N x|^2 })^w e^{-t p_t^w} (e^{-\gamma t^{\alpha} |\xi - Nx|^2 })^w e^{i t D_t \nabla \cdot \nabla} e^{tM_t x \cdot \nabla}  e^{\frac{i}2 (tW_t-  G ) x \cdot x},
\end{equation}
where $\alpha = 2k_0+1$, $N = (G-\transp{G})/2$ denotes the skew-symmetric part of $G$. 
\end{theorem}

Before proving this theorem, we have to introduce some useful tools. First, as in \cite{Vio17,AB21,Ber21}, in order to decompose a semigroup as a product of semigroups we will use the following proposition whose proof relies on properties of Fourier Integral Operators\footnote{More precisely Theorem 5.12 and Proposition 5.9 in \cite{Hor95}.} established by H\"ormander in \cite{Hor95}.
\begin{proposition}[see e.g. Prop 4 in \cite{Ber21}] \label{prop:split} Let $m\geq 1$, $T>0$ and $p_{1,t},\ldots, p_{m+1,t}$ be some complex quadratic forms on $\mathbb{R}^{2n}$ depending continuously on $t \in [0,T)$. If their real part is nonnegative, i.e. $\Reelle p_{k,t} \geq 0$ for $1\leq k\leq m+1$ and $t \in [0,T)$, and if the following decomposition holds
$$
\forall t \in [0,T), \quad e^{-2itJ P_{1,t}}\cdots e^{-2itJ P_{m,t}} = e^{-2itJ P_{m+1,t}},
$$
where $P_{k,t}$ denotes the matrix of $p_{k,t}$ for $1\leq k\leq m+1$, then
$$
\forall t \in [0,T), \quad e^{-t p_{1,t}^w }\cdots e^{-t  p_{m,t}^w} = e^{- tp_{m+1,t}^w}.
$$
\end{proposition}

Thanks to the correspondence given by this proposition, it is sufficient to prove splitting formulas (i.e. decompositions) for exponential of matrices. 
%
As a corollary, we prove the following decomposition which is very useful to study the regularizing effects of semigroups generated by selfadjoint operators. 
\begin{proposition} \label{prop:Strang} Let $n\geq 1$ and $\mathscr{Q}_{2n}(\mathbb{R})$ be the space of real-valued quadratic forms on $\mathbb{R}^{2n}$.

There exists an analytic map $p : \mathscr{O} \to \mathscr{Q}_{2n}(\mathbb{R})$ defined on an open neighborhood of the origin $\mathscr{O}$ in $(\mathscr{Q}_{2n}(\mathbb{R}))^2$, vanishing at the origin and such that for all $ (a,b) \in \mathscr{O}$, we have
$$
0 \leq   5 b \leq    a  \quad \Longrightarrow \quad e^{- b^w } e^{- (p(a,b))^w }  e^{- b^w }  = e^{- a^w }  \quad\mathrm{and} \quad  p(a,b) \geq \frac{a}2.
$$
\end{proposition}
\begin{remark} Roughly speaking, this proposition means that, for quadratic operators, regularizing effects can be added without loss. For example, it implies that the operator $\mathrm{exp}( \partial_x^2 - x^2)$ makes functions as smooth as $\mathrm{exp}( \partial_x^2)$ does. Unfortunately, this property is not true in general: there may be some destructive interactions. For example, being given $m\geq 1$, the operator $\mathrm{exp}( -(-\partial_x^2)^{m} - x^2)$ is only as smoothing as $\mathrm{exp}( -(-\partial_x^2)^{2m/(1+m)})$ (see \cite{Alp20b}).
\end{remark}
\begin{proof}[Proof of Proposition \ref{prop:Strang}] Let $\mathscr{B}_{\varepsilon_1}$ be the centered open ball of radius $\varepsilon_1 := \frac16 \log 2$ in $S_{2n}(\mathbb{R})$ (the space of $n\times n$ symmetric matrices). First, we note that if $A,B \in \mathscr{B}_{\varepsilon_1}$, then
$$
|e^{ 2i J B}e^{- 2i J A}e^{ 2i J B} - I_{2n}| \leq e^{ 4|B| + 2 |A|} - 1 < 1.
$$
Therefore, the following map is well defined on $\mathscr{B}_{\varepsilon_1}^2$
\begin{equation}
\label{eq:defpab}
P(A,B) := (-2 i J)^{-1}\log( e^{ 2i J B}e^{- 2i J A}e^{ 2i J B} ),
\end{equation}
where $\log$ is defined by its power series. First, we note that $P$ is an odd function\footnote{this a well known and useful property in Geometric Numerical Integration which is related to the fact that Strang splittings preserve the reversibility, see \cite{HLW06}.}: 
$$
P(-A,-B)= (-2 i J)^{-1}\log\Big( \big[ e^{ 2i J B}e^{- 2i J A}e^{ 2i J B} \big]^{-1}  \Big) = - P(A,B).
$$
Then, we expand $P(A,B)$ in power series
\begin{equation*}
\begin{split}
P(A,B) &= (-2 i J)^{-1} \sum_{k\geq 1} \frac{(-1)^{k-1}}{k} \big( e^{ 2i J B}e^{- 2i J A}e^{ 2i J B} - I_{2n}\big)^k \\
					&= (-2 i J)^{-1} \sum_{k\geq 1} \frac{(-1)^{k-1}}{k} \bigg(  \sum_{  \alpha+\beta+\gamma >0  } \frac{ (2i J B)^{\alpha}  (-2i J A)^{\beta} (2i J B)^{\gamma}}{\alpha !\beta !\gamma ! } \bigg)^k \\
					&= A - 2B+  \sum_{k\geq 1}     \sum_{ \substack{ (\alpha,\beta,\gamma) \in (\mathbb{N}^3\setminus \{0\})^k \\ |(\alpha,\beta,\gamma) |_1 \geq 2 } }    c^{(k)}_{\alpha,\beta,\gamma} M_{\alpha,\beta,\gamma} (A,B),
		\end{split}
\end{equation*}
where 
$$
M_{\alpha,\beta,\gamma} (A,B)  = J^{-1} (JB)^{\alpha_1} (JA)^{\beta_1} (JB)^{\gamma_1}  \cdots  (JB)^{\alpha_k} (JA)^{\beta_k} (JB)^{\gamma_k},
$$
$$
c^{(k)}_{\alpha,\beta,\gamma} =  \frac{(-1)^{k}}{k} \frac{ (2i)^{-1+|(\alpha,\beta,\gamma) |_1 } (-1)^{-1+|\beta |_1  }  }{\alpha_1 ! \beta_1 ! \gamma_1 ! \cdots \alpha_k ! \beta_k ! \gamma_k ! }  \quad \mathrm{and} \quad  |(\alpha,\beta,\gamma) |_1 = \sum_{j=1}^k  \alpha_j +\beta_j + \gamma_j.
$$
But, since $P$ is odd, this expansion rewrites
\begin{equation}
\label{eq:expP}
P(A,B)  =  A - 2B+   \sum_{k\geq 1}     \sum_{ \substack{ (\alpha,\beta,\gamma) \in (\mathbb{N}^3\setminus \{0\})^k \\ |(\alpha,\beta,\gamma) |_1 \geq 2 } }     \kappa^{(k)}_{\alpha,\beta,\gamma} M_{\alpha,\beta,\gamma} (A,B),
\end{equation}
where 
$$
\kappa^{(k)}_{\alpha,\beta,\gamma} = c^{(k)}_{\alpha,\beta,\gamma} \quad \mathrm{if} \quad |(\alpha,\beta,\gamma) |_1 \quad \mathrm{is \ odd} \quad \mathrm{and} \quad \kappa^{(k)}_{\alpha,\beta,\gamma} = 0 \quad \mathrm{else}.
$$
Since $\kappa^{(k)}_{\alpha,\beta,\gamma}$ is a real number by definition, we deduce that $P(A,B)$ is a real matrix. Then, we note that
$$
\transp{M_{\alpha,\beta,\gamma} (A,B) } = -(-1)^{|(\alpha,\beta,\gamma) |_1}  M_{\overleftarrow{\gamma},\overleftarrow{\beta}, \overleftarrow{\alpha}} (A,B),
$$
where the transformation $\overleftarrow{\cdot}$ inverses the order of the indices, i.e. if $v\in \mathbb{R}^{k}$ then $(\overleftarrow{v})_j = v_{k-j+1}$. But since $\kappa^{(k)}_{\alpha,\beta,\gamma}$ vanishes if $|(\alpha,\beta,\gamma) |_1$ is even and $\kappa^{(k)}_{\alpha,\beta,\gamma} = \kappa^{(k)}_{\overleftarrow{\gamma},\overleftarrow{\beta}, \overleftarrow{\alpha}}$, we deduce from the expansion \eqref{eq:expP} of $P$ that  $P(A,B)$ is a real symmetric matrix. 

\medskip

We define naturally $p(a,b)$ by conjugation: being given two real-valued quadratic forms $a,b$ such that $\| a \|_{L^\infty(\mathbb{S}^{2n-1})}< \varepsilon_1$ and $\| b \|_{L^\infty(\mathbb{S}^{2n-1})}< \varepsilon_1$, we define $p(a,b)$ as the quadratic form associated with  $P(A,B)$, where $A$ and $B$ are the matrices of $a$ and $b$. From now, we consider two quadratic forms $a,b$ such that 
$$
0 \leq   5 b \leq    a < \varepsilon_2 |\cdot|^2,
$$
where $\varepsilon_2\leq \varepsilon_1$ is a universal constant that will be determined later. We aim at proving that $p(a,b) \geq a /2$. Being given $X\in \mathbb{R}^{2n}$, using the expansion \eqref{eq:expP}, we get
\begin{equation}
\label{eq:tempopo}
p(a,b)(X) \geq \Big(1-\frac25\Big)a(X) -  \sum_{k\geq 1}     \sum_{ \substack{ (\alpha,\beta,\gamma) \in (\mathbb{N}^3\setminus \{0\})^k \\ |(\alpha,\beta,\gamma) |_1 \geq 2 } }      |\kappa^{(k)}_{\alpha,\beta,\gamma}|  |\transp{X}M_{\alpha,\beta,\gamma} (A,B) X|.
\end{equation}
Hence we have to estimate $|\transp{X}M_{\alpha,\beta,\gamma} (A,B) X|$. This term is of the form $(L X) \cdot C (R X)$ where $L,R \in \{A,B \} $ and $C$ is a matrix (a product of matrices in $ \{A,B,J \}$). So, by Cauchy-Schwarz' inequality, we have
$$
|\transp{X}M_{\alpha,\beta,\gamma} (A,B) X| \leq |C|  \sqrt{|L| |R|} |\sqrt{L} X| |\sqrt{R}X|. 
$$
Recalling that by assumption $5 b \leq    a< \varepsilon_2 |\cdot|^2$, we deduce that
$$
|\transp{X}M_{\alpha,\beta,\gamma} (A,B) X| \leq |A|^{-1+|(\alpha,\beta,\gamma) |_1 } a(X) \leq \varepsilon_2^{-1+|(\alpha,\beta,\gamma) |_1 } a(X).
$$
Plugging this estimate in \eqref{eq:tempopo}, we get
\begin{equation*}
\begin{split}
p(a,b)(X) &\geq a(X)  \bigg( \frac35 - \sum_{k\geq 1}     \sum_{ \substack{ (\alpha,\beta,\gamma) \in (\mathbb{N}^3\setminus \{0\})^k \\ |(\alpha,\beta,\gamma) |_1 \geq 2 } }      |\kappa^{(k)}_{\alpha,\beta,\gamma}|  \varepsilon_2^{-1+|(\alpha,\beta,\gamma) |_1}    \bigg) \\
&= a(X)  \bigg( \frac35 - \frac1{ (2\varepsilon_2)}\sum_{k\geq 1}   \frac1k  \sum_{ \substack{ (\alpha,\beta,\gamma) \in (\mathbb{N}^3\setminus \{0\})^k \\ |(\alpha,\beta,\gamma) |_1 \geq 2 } }       \frac1{ \alpha_1 ! \beta_1 ! \gamma_1 ! \cdots \alpha_k ! \beta_k ! \gamma_k ! }    (2\varepsilon_2) ^{|(\alpha,\beta,\gamma) |_1}    \bigg) \\
&=  a(X)  \bigg( \frac35 - \frac1{ (2\varepsilon_2)} (-\log(2- e^{2\varepsilon_2}e^{2\varepsilon_2} e^{2\varepsilon_2})-  3(2 \varepsilon_2)) \bigg).
\end{split}
\end{equation*}
Since $3/5>1/2$ and the map $x \mapsto -x^{-1} (\log(2-e^{x}) -x)$ is smooth and vanishes as $x$ goes to $0$, we deduce from the previous estimate that there exists $\varepsilon_2>0$ such that $p(a,b) \geq  a/2$. 

A fortiori, $p(a,b)$ is nonnegative. Consequently, by definition of $p(a,b)$ (see \eqref{eq:defpab}), using the exact classical-quantum correspondance (i.e. Proposition \ref{prop:split}), we get the expected splitting formula: 
$$e^{- b^w } e^{- (p(a,b))^w }  e^{- b^w }  = e^{- a^w }.$$
\end{proof}

Finally, the last tool we need is the Mehler formula which allows to express semigroups as pseudo-differential operators.
\begin{theorem}[Mehler formula, Thm 4.2 in \cite{Hor95}] \label{thm:Meh} Let $q$ be a complex-valued quadratic form on $\mathbb{R}^{2n}$ whose real part is nonnegative and let $Q$ be its matrix. Then, $e^{-q^w}$ is a pseudo-differential operator and whenever the condition $\mathrm{det} \cos( J Q) \neq 0$ is satisfied, the following formula holds
$$
e^{-q^w} = \frac1{\sqrt{\mathrm{det} \cos( J Q)}}( e^{-  m} )^w,
$$ 
where $m$ denotes the complex quadratic form associated with the matrix $J^{-1} \tan (JQ)$.
\end{theorem}
\begin{remark} A priori, the Mehler formula seems ambiguous due to sign indetermination of the square root. Fortunately it is not. Indeed, when $Q$ is small enough it is well defined through the standard holomorphic functional calculus (choosing the principal determination of the square root). Moreover, H\"ormander has proven in \cite{Hor95} (just before Theorem 4.1) that there is a natural way to extend $\sqrt{\mathrm{det} \cos( J Q)}$ as an entire function of $Q$.
\end{remark}

\begin{proof}[Proof of Theorem \ref{thm:the_dec}]
We divide the proof in four steps. First, we are going to explain why we can assume without loss of generality that $G$ is skew-symmetric. Then, we will recall some useful properties of the polar decomposition of $e^{-tq^w}$ established in \cite{AB21}. Finally, in the two last steps, we will decompose separately the unitary part and the symmetric part.

\medskip

\noindent \emph{$\triangleright$ Step 1: Reduction to the skew-symmetric case.} As a consequence of the metaplectic invariance of the Weyl calculus (see \cite[Theorem 18.5.9]{Hor85}), we know that
$$
e^{ -\frac{i}2 G x \cdot x}  e^{-tq^w}  e^{ \frac{i}2 G x \cdot x} = \exp ( -t e^{ -\frac{i}2 G x \cdot x}  q^w  e^{ \frac{i}2 G x \cdot x} )  = e^{-t(q\circ \mathcal{L})^w},
$$
where $\mathcal{L}$ denotes the \emph{Lie transform} (i.e. the Hamiltonian flow at time $1$) of the classical Hamiltonian $h(x,\xi):=-\frac12 G x \cdot x$, i.e. 
$$
\mathcal{L} = \Phi_1 \quad \mathrm{where} \quad \partial_t \Phi_t = J (\nabla h)\circ \Phi_t  \quad \mathrm{and} \quad \Phi_0 = \mathrm{I}_{2n}.
$$
Actually, here, $\mathcal{L}$ is nothing but a symplectic transvection\footnote{also called shear mapping.}: $\mathcal{L}(x,\xi) = (x,\xi + G^{(sym)} x)$
where $G^{(sym)} = (G+\transp{G})/2$ denotes the symmetric part of $G$. 

Let $\widetilde{S}$ be the singular space of $q\circ \mathcal{L}$. In order to reduce the problem to the case where $G$ is skew-symmetric, we just have to prove that $
\widetilde{S} = \mathcal{L}^{-1} S. $
Indeed, since by assumption $S$ is included in the graph of $G$, we would have
$$
\widetilde{S}=\mathcal{L}^{-1} S \subset \mathcal{L}^{-1} \big\{ (x,Gx) \ | \ x\in \mathbb{R}^n \big\} = \big\{ (x,Gx -  G^{(sym)} x ) \ | \ x\in \mathbb{R}^n \big\} =  \big\{ (x,Nx) \ | \ x\in \mathbb{R}^n \big\},
$$
and so the decomposition of $e^{-t(q\circ \mathcal{L})^w}$ would provide the one of $e^{-tq^w}$ (because we will also check that the global indexes of the spaces $S$ and $\widetilde{S}$ are the same).

Now let us prove that $\widetilde{S} = \mathcal{L}^{-1} S $. By construction, since $\mathcal{L}$ is symplectic, it satisfies the relation $\mathcal{L}^{-1} J = J \transp{\mathcal{L}}$. Therefore, for all $k\geq 0$, we have
$$
(\Imag (J \transp{\mathcal{L}} Q \mathcal{L}))^k = ( J \transp{\mathcal{L}} (\Imag Q) \mathcal{L})^k = (\mathcal{L}^{-1} J  (\Imag Q) \mathcal{L} )^k = \mathcal{L}^{-1} ( J  (\Imag Q) )^k \mathcal{L},
$$
and so, since $\transp{\mathcal{L}} Q \mathcal{L}$ is the matrix of $q \circ \mathcal{L}$, we have 
\begin{equation*}
\begin{split}
\widetilde{S} &= \bigcap_{k \in \mathbb{N}} \mathrm{Ker} \, [\Reelle \transp{\mathcal{L}}  Q   \mathcal{L} ] (\Imag (J \transp{\mathcal{L}} Q \mathcal{L}))^k  \\
&=\bigcap_{k \in \mathbb{N}} \mathrm{Ker} \, \transp{\mathcal{L}} (\Reelle Q) ( J  (\Imag Q) )^k \mathcal{L} \\ & = \bigcap_{k \in \mathbb{N}} \mathcal{L}^{-1}\mathrm{Ker} \transp{\mathcal{L}} (\Reelle Q) ( J  (\Imag Q) )^k \\
& = \bigcap_{k \in \mathbb{N}} \mathcal{L}^{-1}\mathrm{Ker} (\Reelle Q) ( J  (\Imag Q) )^k = \mathcal{L}^{-1}S.
\end{split}
\end{equation*} 
Moreover, we point out that, as a consequence of these identities, the global indices (i.e. $k_0$) of $S$ and $\widetilde{S}$ are equal. Thanks to this reduction, we will now only treat the skew-symmetric case (i.e. $G=N$ and $q\circ \mathcal{L}=q$). 

\medskip

\noindent \emph{$\triangleright$ Step 2: Some reminders about the polar decomposition of $e^{-tq^w}$.} The properties of the polar decomposition of $e^{-tq^w}$ have been studied in details by the authors in \cite{AB21}. More precisely, in \cite[Theorem 2.1]{AB21}, it is proven that there exist $T>0$ and two families of real quadratic forms $a_t,b_t$ on $\mathbb{R}^{2n}$, depending analytically on $t \in (-T,T)$, such that for all $t\in (0,T)$, $a_t$ is nonnegative and we have
\begin{equation}
\label{eq:polardec}
e^{-t q^w} = e^{- t a_t^w} e^{- i t b_t^w}.
\end{equation}
Since the quadratic forms are real-valued, $e^{- i t b_t^w}$ is unitary on $L^2$ while $e^{- t a_t^w}$ is self-adjoint.
Moreover, Theorem 2.2 of \cite{AB21} provides\footnote{see also \cite[Lemmata 7.13 and 4.1]{AB21} for further details.} the following quantitative estimate on $a_t$:
\begin{equation}
\label{eq:est_at}
\forall t\in (-T,T),\forall X \in \mathbb{R}^{2n}, \quad  a_t(X) \gtrsim t^{2k_0}|\Pi_{S^\perp} X|^2,
\end{equation}
where $\Pi_{S^\perp} $ denotes the orthogonal projection on $S^\perp$ (actually this estimate is sharp in the sense that it could be proven that $a_t$ vanishes on $S$). 

\medskip

\noindent \emph{$\triangleright$ Step 3: Decomposition of the unitary part.}  Now, we aim at proving that there exist $t_0>0$ and some matrices $D_t,M_t,W_t$ (as in Theorem \ref{thm:the_dec}) such that for all $t\in(-t_0,t_0)$ we have
\begin{equation}
\label{dec:btop}
e^{-itb_t^w} = e^{- \frac{t}2\mathrm{tr} M_t}e^{i t D_t \nabla \cdot \nabla} e^{tM_t x \cdot \nabla}  e^{\frac{i}2 tW_t  x \cdot x}.
\end{equation}
As a consequence of the classical-quantum exact correspondence (i.e. Proposition \ref{prop:split}), since $M_t x \cdot \nabla - \frac12 \mathrm{tr} M_t = -i (M_tx \cdot \xi)^w$, it is sufficient to prove a factorization of the form
\begin{equation}
\label{dec:Btmat}
e^{2 tJ B_t} = \exp \left(-2tJ\begin{pmatrix} 0 & 0 \\ 0 & D_t\end{pmatrix} \right) \exp \left(-tJ\begin{pmatrix} 0 & \transp{M_t} \\ M_t & 0\end{pmatrix} \right) \exp \left(tJ\begin{pmatrix} W_t & 0 \\ 0 & 0\end{pmatrix} \right),
\end{equation}
where $B_t$ is the matrix of $b_t$. 

Actually it is a quite direct application of the Local Inversion Theorem. Indeed, there exists a neighborhood of the origin $\mathcal{N}$ in $\mathfrak{sp}_{2n}(\mathbb{R}) := J S_{2n}(\mathbb{R})$ (where $S_{2n}(\mathbb{R})$ denotes the space of real symmetric matrices) such that the following map $\Psi$ is well defined on $\mathcal{N}$
$$
\Psi(J K) = \log\left(\exp \left(J\begin{pmatrix} 0 & 0 \\ 0 & K_{2,2}\end{pmatrix} \right) \exp \left(J\begin{pmatrix} 0 & K_{1,2} \\ K_{2,1} & 0\end{pmatrix} \right) \exp \left(J\begin{pmatrix} K_{1,1} & 0 \\ 0 & 0\end{pmatrix} \right)\right),
$$
where $\log$ is defined by its power series and $K = \begin{pmatrix} K_{1,1} & K_{1,2} \\ K_{2,1} & K_{2,2}\end{pmatrix}$ denotes the decomposition of $K$ by blocks of size $n\times n$ (note that since $K$ is symmetric, it satisfies $K_{1,2} = \transp{K_{2,1}}$). 
We note that, as a consequence of the Baker-Campbell-Hausdorff formula (see e.g. Theorem 10 in \cite{BCOR09}), since $\mathfrak{sp}_{2n}(\mathbb{R})$ is a Lie algebra, $\Psi$ maps $\mathcal{N}$ into $\mathfrak{sp}_{2n}(\mathbb{R})$. Moreover, since both the differential of the exponential at the origin and of the logarithm at the identity are equal to the identity, we deduce that the differential of $\Psi$ at the origin is equal to the identity (and so that it is invertible). Finally, since $t\mapsto B_t$ is a smooth map, the existence of matrices $D_t,W_t,M_t$, depending smoothly on $t$ and satisfying, provided that $|t|$ is small enough, the decomposition \eqref{dec:Btmat} is just a consequence of the Local Inversion Theorem applied to $\Psi$ at the origin.

\medskip

\noindent \emph{$\triangleright$ Step 4:  Decomposition of the symmetric part.} Now, we focus on the decomposition of $e^{-ta_t^w}$. We aim at proving that there exist $\gamma,t_0>0$,  $p_t$ (as in Theorem \ref{thm:the_dec}) and some constants $\mathfrak{c}_t$ depending continuously on $t$ such that for all $t\in[0,t_0)$,
\begin{equation}
\label{eq:decsympart}
e^{-ta_t^w} = \mathfrak{c}_t \, (e^{-\gamma t^\alpha |\xi -  N x|^2 })^w e^{-t p_t^w} (e^{-\gamma t^{\alpha} |\xi - Nx|^2 })^w.
\end{equation}

\medskip

\noindent \emph{$\triangleright$ Substep 4.1: Inversion of the Mehler formula.}
First, in order to get such a decomposition using the classical-quantum exact correspondence (i.e. Proposition \ref{prop:split}), we aim at rewriting the twisted diffusion $(e^{-\gamma t^\alpha |\xi -  N x|^2 })^w$ as an evolution operator thanks to the Mehler formula (i.e. Theorem \ref{thm:Meh}). 

Using the holomorphic functional calculus, provided that $|s|$ is small enough to avoid singularities, we define the matrices $R_{s}$, with $s\in \mathbb{R}$, by
$$
R_{s} := (sJ)^{-1} \arctan\left( s J \mathfrak{N} \right) \quad \mathrm{where} \quad \mathfrak{N}:= \begin{pmatrix} N^2 & N \\ -N & I_n\end{pmatrix}.
$$
First, we notice that since $\arctan$ vanishes at the origin, $R_s$ depends smoothly (analytically) on $s$. Then, since $N$ is skew-symmetric, we deduce that $\mathfrak{N}$ is symmetric. Therefore, expanding $\arctan$ in power series (provided that $s$ is small enough), we get 
\begin{equation}
\label{eq:trick_arctan}
R_s = (sJ)^{-1}\sum_{k=0}^{+\infty} (-1)^{k} \frac{(sJ\mathfrak{N})^{2k+1}}{2k+1} = \sum_{k=0}^{+\infty} \frac{\transp{\mathfrak{F}_s^k  } \, \mathfrak{N} \, \mathfrak{F}_s^k }{2k+1} \quad \mathrm{with} \quad \mathfrak{F}_s := sJ\mathfrak{N}.
\end{equation}
It appears from this formula that $R_s$ is a real symmetric matrix. We denote by $r_s$ (resp. $\mathfrak{n}$) the quadratic form of which $R_s$ (resp. $\mathfrak{N}$) is the matrix. Since $\mathfrak{n} (x,\xi) = |\xi - Nx|^2$ is nonnegative, we deduce from \eqref{eq:trick_arctan} that
$$
r_s = \mathfrak{n} + \sum_{k \geq 1} \frac{\mathfrak{n} \circ \mathfrak{F}_s^k}{2k+1} \geq  \mathfrak{n}.
$$
Therefore, $r_s$ is nonnegative. Hence, as a consequence of the Mehler formula (see Theorem \ref{thm:Meh}), provided that $s\geq 0$ is small enough, we can write the twisted diffusion as an evolution operator
\begin{equation}
\label{eq:Mehlerinverse}
e^{- s r_s^w} =  \frac1{\sqrt{\mathrm{det} \cos( sJ R_s)}} (e^{-s |\xi -  N x|^2 })^w.
\end{equation}
Finally, we aim at establishing an upper bound for $r_s$. First, we note that
$$
\forall X \in \mathbb{R}^{2n}, \quad |\mathfrak{F}_s X| \leq |s| |\sqrt{\mathfrak{N}}| |\sqrt{\mathfrak{N}} X| =  |s| |\sqrt{\mathfrak{N}}| \sqrt{\mathfrak{n}( X)}.
$$
Therefore, as a consequence of \eqref{eq:trick_arctan}, we have that for all $X \in \mathbb{R}^{2n}$ and $|s| < 2^{-1/2} |\mathfrak{N}|^{-1}$,
\begin{equation}
\label{eq:estrs}
r_s(X) \leq \mathfrak{n}(X) \sum_{k=0}^\infty \frac{ |s|^{2k}  |\mathfrak{N}|^{2k }  }{2k+1} \leq \frac{\mathfrak{n}(X)}{1 - |s \mathfrak{N}|^2} \leq 2 \mathfrak{n}(X).
\end{equation}

\medskip

\noindent \emph{$\triangleright$ Substep 4.2: Strang splitting.} First, we note that since $S$ is included in the graph of $N$, we have the estimate
$\mathfrak{n} \lesssim |\Pi_{S^\perp} \cdot|^2.$
Moreover, as explained in \eqref{eq:est_at}, we know that $ t^{\alpha } |\Pi_{S^\perp} \cdot|^2 \lesssim t a_t$. Consequently, there exists $\gamma>0$ such that for all $t\in [0,T)$, we have
$ 10 \gamma   t^{\alpha } \mathfrak{n}  \leq t a_t$. Moreover, we know from \eqref{eq:estrs} that if $t < (\gamma^{-1} 2^{-1/2} |\mathfrak{N}|^{-1})^{1/\alpha} $, then $r_{\gamma t^{\alpha}} \leq 2 \mathfrak{n}$. Therefore, we have
\begin{equation}
\label{eq:carteRU}
0\leq t < t_1 \quad \Longrightarrow \quad 5 (\gamma t^{\alpha}r_{\gamma t^{\alpha}}) \leq t a_t,
\end{equation}
where $t_1:= \min (T,(\gamma^{-1} 2^{-1/2} |\mathfrak{N}|^{-1})^{1/\alpha})$. Now, we apply Proposition \ref{prop:Strang} in order to decompose $e^{-ta_t^w}$: we get the map $p$ defined on a neighborhood $\mathscr{O}$ of the origin in the space of couples of real-valued quadratic forms on $\mathbb{R}^{2n}$. First, we note that since $t \mapsto a_t$ is smooth, $t a_t$ vanishes as $t$ goes to $0$. Therefore, there exists $t_0 \in (0,t_1)$ such that the following map $p_t:= t^{-1}p(t a_t,\gamma t^{\alpha}r_{\gamma t^{\alpha}})$ is well defined for $t\in (-t_0,t_0)$ and is analytic. Moreover, thanks to the estimate  \eqref{eq:carteRU}, we have that for all $t\in [0,t_0)$,
$$
e^{-  \gamma t^{\alpha}r_{\gamma t^{\alpha}}^w} e^{-p_t^w} e^{-  \gamma t^{\alpha}r_{\gamma t^{\alpha}}^w} = e^{-ta_t^w} \quad \mathrm{and} \quad p_t \geq a_t/2 \geq 0.
$$
Finally, since we have shown in \eqref{eq:Mehlerinverse} that $e^{-  \gamma t^{\alpha}r_{\gamma t^{\alpha}}^w} $ is a pseudo-differential operator whose symbol is $ \sqrt{\mathrm{det} \cos(   \gamma t^{\alpha}J R_ {\gamma t^{\alpha}})}^{-1} e^{-  \gamma t^{\alpha} \mathfrak{n} }$, we have factorized $e^{-tq^w}$ as expected in \eqref{eq:decsympart}.

\end{proof}

\section{Integral transforms with Gaussian kernels}\label{sec:Gauss}

In this section, we aim at proving $L^{\mathfrak p}$ bounds for the evolution operators involved in the decomposition \eqref{eq:the_dec} given by Theorem \ref{thm:the_dec}. More precisely, the main results of this section are Theorem \ref{thm:fifou}, in which we prove that $e^{- tp_t^w}$ maps $L^{\mathfrak{q}}$ into itself, and Corollary \ref{cor:twisted_vs_schro}, where we prove that $(e^{-\gamma t^{\alpha} |\xi - Nx|^2 })^w e^{i t D_t \nabla \cdot \nabla}$ maps $L^{\mathfrak{p}}$ into $L^{\mathfrak{q}}$.

To prove these results, we shall write evolution operators as integral transforms with Gaussian kernels. Naturally, we will have to compute Fourier transforms of complex Gaussians, so we recall the following classical result.
\begin{proposition}[see e.g. Thm 1 in \cite{Fol89} page 256] \label{prop:TFGauss} Let $A$ be an $n\times n$ complex symmetric matrix whose real part is positive-definite. Then for any $z\in \mathbb{C}^n$, 
$$
\int_{\mathbb{R}^n} e^{- \frac12 Ax \cdot x } e^{i x \cdot z}\, \mathrm{d}x = \frac{(2\pi)^{n/2}}{\sqrt{\mathrm{det} A}} e^{- \frac12 A^{-1}z \cdot z  },
$$
where the branch of the square root is determined by the requirement that $\sqrt{\mathrm{det}\,  A}>0$ when $A$ is real and positive-definite.
\end{proposition}

As a corollary, we identity the pseudo-differential operators with non-degenerate Gaussian symbols as Gaussian integral transforms.
\begin{corollary} 
\label{cor:g} Let $m$ be a real-valued nonnegative quadratic on $\mathbb{R}^{2n}$ of the form
$$
m(x,\xi) = \frac12 r(x) + L x \cdot \xi + \frac12 b(\xi),
$$
where $L$ is a real $n \times n$ matrix, $r,b$ are some real-valued quadratic forms on $\mathbb{R}^n$ and $b$ is positive-definite. Then for all $u\in \mathscr{S}(\mathbb{R}^n)$, 
$$
(e^{- m})^w u(x) = \int_{\mathbb{R}^n} g(x,y) \, u(y)\, \mathrm{d}y,
$$
where
\begin{equation}
\label{eq:ouaichGG}
	g(x,y) = \frac{(2\pi)^{-n/2}}{\sqrt{\mathrm{det}\,  B}}  \mathrm{exp}\bigg(-  \frac12 k\Big(\frac{x+y}2\Big) - \frac12 B^{-1} (x-y) \cdot (x-y) - i (x-y) \cdot B^{-1}L\Big(\frac{x+y}2\Big) \bigg),
\end{equation}
 $B$ (resp. $R$) denoting the matrix of $b$ (resp. $r$) and $k$, being the real-valued quadratic form of matrix $K=R - \transp{L}B^{-1}L$, is nonnegative.
\end{corollary}
\begin{proof}
By definition of the Weyl quantization, we formally have that for all $u\in\mathscr S(\mathbb R^n)$,
$$
(e^{- m})^w u(x) = (2\pi)^{-n} \iint_{\mathbb{R}^{2n}}e^{i (x-y) \cdot \xi} e^{-  m(\frac{x+y}2 , \xi )} u(y)\, \mathrm{d}y \,  \mathrm{d}\xi.
$$
Moreover, since the quadratic form $b$ is positive-definite and $u\in L^1(\mathbb R^n)$, the above integral is a well-defined Lebesgue integral which coincides with the function $(e^{- m})^wu$. The same arguments allow to apply Fubini's theorem to permute the integrals. Therefore, we get
$$
(e^{- m})^w u(x) = \int_{\mathbb{R}^n} g(x,y) \, u(y)\, \mathrm{d}y \quad\mathrm{where } \quad
g(x,y) =  (2\pi)^{-n} \int_{\mathbb{R}^n}e^{i (x-y) \cdot \xi} e^{-  m(\frac{x+y}2 , \xi )}\,   \mathrm{d}\xi.
$$
Thanks to the block decomposition of $m$, $g$ can be rewritten as a Fourier transform
$$
g(x,y) =  (2\pi)^{-n} e^{-\frac12 r(\frac{x+y}2)}\int_{\mathbb{R}^n} e^{- \frac12 b(\xi)} e^{i (x-y + iL(\frac{x+y}2)) \cdot \xi}\,   \mathrm{d}\xi.
$$
Therefore, we have
$$
g(x,y) =  (2\pi)^{-n/2} \frac1{\sqrt{\mathrm{det}\, B}} e^{-\frac12 r(\frac{x+y}2)}  e^{- \frac12 B^{-1} (x-y + iL(\frac{x+y}2)) \cdot (x-y + iL(\frac{x+y}2)) },
$$
and so, expanding the exponent, we get \eqref{eq:ouaichGG}. Finally, putting $m$ in canonical form with respect to $\xi$, i.e.
$$
m(x,\xi) = \frac12 k(x) + \frac12 b(\xi + B^{-1} Lx),
$$
we note that since $m$ is nonnegative, then $k$ is also nonnegative.
\end{proof}

It will be very useful to apply this result to express twisted diffusion operators as integral transforms.
\begin{corollary} \label{cor:Ng}
Let $N$ be a real $n\times n$ skew-symmetric matrix. Then, for all $\varepsilon >0$ and all $u\in \mathscr{S}(\mathbb{R}^n)$, we have
$$
(e^{-\frac{\varepsilon}2 |\xi - Nx|^2 })^w u(x) =(2\pi \varepsilon )^{-n/2}  \int_{\mathbb{R}^n} e^{-\frac1{2\varepsilon} |x-y|^2 + i (x-y) \cdot N x }u(y)  \, \mathrm{d}y.
$$
\end{corollary}  
\begin{proof} It is sufficient to apply Corollary \ref{cor:g} and to note that since $N$ is skew-symmetric, we have 
$$
(x-y) \cdot N \Big(\frac{x+y}2\Big) = (x-y) \cdot N\Big(\frac{x+y + x-y}2\Big) =  (x-y) \cdot N x.
$$
\end{proof}

\begin{theorem} \label{thm:fifou}For all $1\leq \mathfrak{q}\leq +\infty$, all $u\in \mathscr{S}(\mathbb{R}^n)$ and all real-valued nonnegative quadratic form $a$ on $\mathbb{R}^{2n}$,  we have
$$
\| e^{- a^w} u \|_{L^\mathfrak{q}} \leq   \|  u \|_{L^\mathfrak{q}} .
$$
\end{theorem}
\begin{proof} Since $a \mapsto e^{- a^w} u \in \mathscr{S}(\mathbb{R}^n)$ is a $C^\infty$ map \cite[Theorem 4.2 page 426]{Hor95}, by density, we only have to deal with the case where $a$ is positive-definite. So we assume that $a$ is positive-definite and we denote by $A$ its matrix. 
First, we note that since
$JA = \sqrt{A}^{-1} (\sqrt{A} J \sqrt{A}) \sqrt{A}$, $JA$ is conjugated to a skew-symmetric matrix and so its spectrum is purely imaginary. As a consequence, the eigenvalues of $\cos JA$ are all larger than or equal to $1$ and so $|\sqrt{\mathrm{det} \cos JA}|\geq 1$. Therefore, as a consequence of the Mehler formula (see Theorem \ref{thm:Meh}), we have
$$
\| e^{- a^w} u \|_{L^\mathfrak{q}} \leq \| (e^{- m})^w u \|_{L^\mathfrak{q}},
$$
where $m$ is the real-valued quadratic form associated with the matrix $J^{-1} \tan (JA)$. Moreover, since $a$ is positive-definite, it can be checked (see e.g. \cite[Theorem 4.2 page 426]{Hor95} or \cite[Lemma 3.1]{Alp20a}) that $m$ is also positive-definite. Now, we decompose $m$ by blocks as
$$
m(x,\xi) = \frac12r(x) + L x \cdot \xi + \frac12b(\xi),
$$
where $L$ is a real $n \times n$ matrix and $r,b$ are some real-valued quadratic forms on $\mathbb{R}^n$. Note that, since $m$ is positive-definite, both $r$ and $b$ are positive-definite. We denote by $B$ the matrix of $b$. Thanks to Corollary \ref{cor:g}, $(e^{- m})^w$ can be written as a Gaussian integral transform whose kernel $g$ is given by \eqref{eq:ouaichGG}. Therefore, applying the triangular inequality, using that $k$ (the quadratic form defined in Corollary \ref{cor:g}) is nonnegative, we have
$$
|(e^{- m})^w u(x)| \leq \frac{ (2\pi)^{-n/2}}{\sqrt{\mathrm{det}\, B}} \int_{\mathbb{R}^n} e^{- \frac12 B^{-1} (x-y) \cdot (x-y) } |u(y)| \, \mathrm{d}y.
$$
Hence, by applying Young's convolution inequality, we get
$$
\| e^{- a^w} u \|_{L^\mathfrak{q}} \leq \| (e^{- m})^w u \|_{L^\mathfrak{q}} \leq \frac{ (2\pi)^{-n/2}}{\sqrt{\mathrm{det}\, B}} \big\| e^{- \frac12 B^{-1} y \cdot y  } \big\|_{L^1}  \|  u \|_{L^\mathfrak{q}} = \|  u \|_{L^\mathfrak{q}}.
$$
\end{proof}

Now, we focus on the interactions between a twisted diffusion (i.e. an operator of the form $(e^{-t |\xi - Nx|^2 })^w$ with $N$ skew-symmetric) and a dispersion (i.e. an operator of the form $e^{i t d(\nabla)  }$ with $d$ a real-valued quadratic form). More precisely, in the following proposition (which is the main result of this section), we prove that the product of such operators enjoys some diffusion properties.
\begin{proposition} \label{prop:miraculous} Let $\varepsilon>0$, $N$ be a skew-symmetric $n\times n$ real matrix and $d$ be a real-valued quadratic form on $\mathbb{R}^n$. Then, for all $u\in \mathscr{S}(\mathbb{R}^n)$ and $x\in \mathbb{R}^n$, we have
$$
\big| (e^{-\frac{\varepsilon}2 |\xi - Nx|^2 })^w e^{\frac{i}2  d(\nabla)  }u(x) \big| \leq \frac{(2\pi)^{-n/2}}{\sqrt[4]{\mathrm{det} \, (\varepsilon^2 I_{n} + D^2 )}} (e^{- \varepsilon \frac{r}2}  \ast |u| )( x - DNx  ),
$$ 
where $D$ denotes the matrix of $d$ and $r$ is the real-valued quadratic form of matrix $( \varepsilon^2 I_{n} + D^2)^{-1}$.
\end{proposition}
\begin{proof}[Proof of Proposition \ref{prop:miraculous}] Let $x\in \mathbb{R}^n$ be fixed and $u\in \mathscr{S}(\mathbb{R}^n)$. In Corollary \ref{cor:Ng}, we have proven that
$$
 (e^{-\frac{\varepsilon}2 |\xi - Nx|^2 })^w e^{\frac{i}2  d(\nabla)  }u(x) = (2\pi \varepsilon )^{-n/2} ( f_x  \ast e^{\frac{i}2  d(\nabla)  }u )(x)  \quad \mathrm{where} \quad f_x(y) = e^{-\frac1{2\varepsilon} |y|^2 + i y \cdot N x }.
 $$
We define the Fourier transform by
$$
\forall v\in \mathscr{S}(\mathbb{R}^n),\forall \xi\in \mathbb{R}^n, \quad \mathscr{F} v(\xi) := \int_{\mathbb{R}^n} e^{-i\xi \cdot y}v(y) \, \mathrm{d}y.
$$
With this convention, the inverse Fourier transform is given by
$$
\mathscr{F}^{-1} v(y) := (2\pi)^{-n}\int_{\mathbb{R}^n} e^{iy \cdot \xi}v(\xi) \, \mathrm{d}\xi.
$$
Therefore, thanks to the convolution identity, we have
$$
 f_x  \ast e^{\frac{i}2  d(\nabla)  }u  = \mathscr{F}^{-1} [(\mathscr{F}  f_x )  (\mathscr{F} e^{\frac{i}2  d(\nabla)  }u )  ] = \mathscr{F}^{-1} [(\mathscr{F}  f_x )  e^{-\frac{i}2  d  } \mathscr{F} u   ] =(\mathscr{F}^{-1} [(\mathscr{F}  f_x )  e^{-\frac{i}2  d  } ] ) \ast u,
$$
and so
\begin{equation}
\label{eq:hehe}
\big| (e^{-\frac{\varepsilon}2 |\xi - Nx|^2 })^w e^{\frac{i}2  d(\nabla)  }u(x) \big| \leq (2\pi \varepsilon )^{-n/2} \big( |\mathscr{F}^{-1} [(\mathscr{F}  f_x )  e^{-\frac{i}2  d  } ] | \ast |u| \big)(x).
\end{equation}
As a consequence, we aim at estimating $ |\mathscr{F}^{-1} [(\mathscr{F}  f_x )  e^{-\frac{i}2  d  } ] |$.
First, by Proposition \ref{prop:TFGauss}, we have
$$
\mathscr{F}  f_x (\xi)=   \int_{\mathbb{R}^n} e^{-i\xi \cdot y}e^{-\frac1{2\varepsilon} |y|^2 + i y \cdot N x } \, \mathrm{d}y = (2\pi \varepsilon)^{n/2} e^{- \frac{\varepsilon}2 |\xi - N x |^2 }.
$$
We therefore deduce that
\begin{equation*}
\begin{split}
\mathscr{F}^{-1} [(\mathscr{F}  f_x )  e^{-\frac{i}2  d  } ](y) =& \Big( \frac{\varepsilon}{2\pi} \Big)^{n/2} \int_{\mathbb{R}^n} e^{iy \cdot \xi}e^{- \frac{\varepsilon}2 |\xi - N x |^2 }  e^{-\frac{i}2  d(\xi)  }\, \mathrm{d}\xi \\ 
=& \Big( \frac{\varepsilon}{2\pi} \Big)^{n/2} e^{- \frac{\varepsilon}2 |Nx|^2} \int_{\mathbb{R}^n} e^{i(y-i\varepsilon Nx) \cdot \xi}e^{- \frac{1}2 (\varepsilon|\cdot|^2+id)(\xi)} \, \mathrm{d}\xi \\
=& \frac{\varepsilon^{n/2}}{\sqrt{\mathrm{det} \, (\varepsilon I_{n} + iD)}} e^{- \frac{\varepsilon}2 |Nx|^2} e^{- \frac12  \transp{(y-i\varepsilon Nx)} (\varepsilon I_{n} + i D)^{-1} (y-i\varepsilon Nx) },
\end{split}
\end{equation*}
and so, since $(\varepsilon I_{n} + i D)(\varepsilon I_{n} - i D)=\varepsilon^2 I_{n} +  D^2$,
\begin{equation}
\label{eq:tempolopopo}
|\mathscr{F}^{-1} [(\mathscr{F}  f_x )  e^{-\frac{i}2  d  } ](y)| = \frac{\varepsilon^{n/2}}{\sqrt[4]{\mathrm{det} \, (\varepsilon^2 \mathrm{I}_{n} + D^2 )}}  e^{- \frac{\varepsilon}2 k(x,y)},
\end{equation}
where
$$
k(x,y) := \varepsilon^{-1} \Reelle( \transp{(y-i\varepsilon Nx)} (\varepsilon I_{n} + i D)^{-1} (y-i\varepsilon  Nx) ) +  |Nx|^2.
$$
Then, we have to compute this real part. First, we note that
$$
 (\varepsilon I_{n} + i D)^{-1} =  \varepsilon(\varepsilon^2 I_{n} +  D^2)^{-1} - i D(\varepsilon^2 I_{n} +  D^2)^{-1}.
$$
Consequently, we have
\begin{equation*}
\begin{split}
k(x,y) &= \transp{y}(\varepsilon^2 I_{n} +  D^2)^{-1} y -\varepsilon^2 \transp{(Nx)}(\varepsilon^2 I_{n} +  D^2)^{-1} Nx + |Nx|^2 -2  \transp{y} \cdot D(\varepsilon^2 I_{n} +  D^2)^{-1} Nx \\
&= \transp{y}(\varepsilon^2 I_{n} +  D^2)^{-1} y +\transp{(Nx)}D^2 (\varepsilon^2 I_{n} +  D^2)^{-1} Nx  -2  \transp{y} \cdot D(\varepsilon^2 I_{n} +  D^2)^{-1} Nx \\
&= \transp{(y+ DNx)}(\varepsilon^2 I_{n} +  D^2)^{-1} (y- DNx).
\end{split}
\end{equation*}
Hence, to conclude, we just have to plug this formula into \eqref{eq:tempolopopo} and then into \eqref{eq:hehe}.
\end{proof}
As a corollary of this proposition, applying Young's convolution inequality, we deduce the following quantitative estimate.
\begin{corollary} \label{cor:twisted_vs_schro}Let $\varepsilon>0$, $N$ be a skew-symmetric $n\times n$ real matrix and $d$ be a real-valued quadratic form on $\mathbb{R}^n$. Then, we have that for all $u\in \mathscr{S}(\mathbb{R}^n)$ and $1\leq \mathfrak{p} \leq \mathfrak{q} \leq \infty$,
\begin{equation}
\label{eq:recip_Guardia}
\big\| (e^{-\frac{\varepsilon}2 |\xi - Nx|^2 })^w e^{\frac{i}2  d(\nabla)  }u \big\|_{L^\mathfrak{q}} \lesssim_n \varepsilon^{-\frac{n}{2\mathfrak{r}}}  \, (\mathrm{det} \, (\varepsilon^2 I_{n} + D^2 ))^{\frac{1}{2\mathfrak{r}} - \frac14} \,  \| u\|_{L^\mathfrak{p}},
\end{equation}
where $D$ denotes the matrix of $d$ and $\mathfrak{r} = (1- \mathfrak{p}^{-1} + \mathfrak{q}^{-1} )^{-1} \in [1,+\infty]$.
\end{corollary}
\begin{proof} By Proposition \ref{prop:miraculous}, we have
$$
\big\| (e^{-\frac{\varepsilon}2 |\xi - Nx|^2 })^w e^{\frac{i}2  d(\nabla)  }u \big\|_{L^\mathfrak{q}} \leq \frac{(2\pi)^{-n/2}}{\sqrt[4]{\mathrm{det} \, (\varepsilon^2 I_{n} + D^2 )}}\big\| (e^{- \varepsilon \frac{r}2}  \ast |u| )\circ (I_n - DN) \big\|_{L^\mathfrak{q}}.
$$
Hence a change of variable provides the estimate
$$
\big\| (e^{-\frac{\varepsilon}2 |\xi - Nx|^2 })^w e^{\frac{i}2  d(\nabla)  }u\big\|_{L^\mathfrak{q}} \leq \frac{(2\pi)^{-n/2}}{\sqrt[4]{\mathrm{det} \, (\varepsilon^2 I_{n} + D^2 )} \, |\mathrm{det}(I_n - DN)|^{\frac1{\mathfrak{q}}} } \big\| e^{- \varepsilon \frac{r}2}  \ast |u|  \big\|_{L^\mathfrak{q}}.
$$
Then we note that $|\mathrm{det}(I_n - DN)|\geq 1$. Indeed, since $D$ is invertible and symmetric, $DN$ is conjugated to a skew-symmetric matrix\footnote{because $DN = \sqrt{D} \sqrt{D} N \sqrt{D} \sqrt{D}^{-1}$.}, and so its spectrum is purely imaginary. As a consequence, $|\mathrm{det}(I_n - DN)|$ is a product of factors of the form $|1+i\lambda|$ with $\lambda \in \mathbb{R}$ which are all larger than $1$.

Therefore, applying Young's convolution inequality, we get
$$
\big\| (e^{-\frac{\varepsilon}2 |\xi - Nx|^2 })^w e^{\frac{i}2  d(\nabla)  }u \big\|_{L^\mathfrak{q}} \leq \frac{(2\pi)^{-n/2}}{\sqrt[4]{\mathrm{det} \, (\varepsilon^2 I_{n} + D^2 )}  } \| e^{- \varepsilon \frac{r}2}    \|_{L^\mathfrak{r}} \| u \|_{L^\mathfrak{p}},
$$
where $\mathfrak{r} = (1- \mathfrak{p}^{-1} + \mathfrak{q}^{-1} )^{-1} \in [1,+\infty]$ is well-defined thanks to the assumption that $1\leq \mathfrak{p} \leq \mathfrak{q} \leq \infty$. Finally, recalling that $( \varepsilon^2 I_{n} + D^2)^{-1}$ is the matrix of $r$, we deduce from Proposition \ref{prop:TFGauss} that (if $\mathfrak{r} \neq +\infty$)\footnote{If $\mathfrak{r} = +\infty$, we just use that $\| e^{- \varepsilon \frac{r}2}    \|_{L^\mathfrak{r}} = 1$.},
$$
\| e^{- \varepsilon \frac{r}2}    \|_{L^\mathfrak{r}} =\bigg( \int_{\mathbb{R}^n} e^{- \varepsilon \mathfrak{r} \frac{r(x)}2} \, \mathrm{d}x  \bigg)^{\frac1{\mathfrak{r}}} = \Big(  \Big(\frac{2\pi}{\varepsilon \mathfrak{r}}\Big)^{n/2} \sqrt{\mathrm{det} \, (\varepsilon^2 I_{n} + D^2 )}  \Big)^{\frac1{\mathfrak{r}}},
$$
which provides the estimate \eqref{eq:recip_Guardia} we aimed at proving. 
\end{proof}

\section{Local smoothing effects}\label{sec:localsmooth}

This section is devoting to the proof of the main result of this paper, namely Theorem \ref{thm:maindir}.  This result deals with the local smoothing effects and the gains of integrability of semigroups generated by accretive quadratic operators whose singular spaces are assumed to be included in the graph of a real $n \times n$ matrix. To that end, we will use the decomposition \eqref{eq:the_dec} established in Theorem \ref{thm:the_dec}, and also the gain of integrability results established in Section \ref{sec:Gauss}.

First of all, we focus on the case of twisted diffusions:

\begin{proposition}\label{prop:smoothtwist} Let $N$ be a real $n\times n$ skew-symmetric matrix. Then, for all $1\le\mathfrak q\le\infty$, there exists a positive constant $C>1$ such that for all $m\geq0$, $\varepsilon>0$ and $u\in\mathscr S(\mathbb R^n)$,
\begin{equation}\label{21102021E1}
	\big\Vert\langle Nx\rangle^{-m}\mathrm d^m((e^{-\frac{\varepsilon}2\vert\xi - Nx\vert^2})^wu)\big\Vert_{L^{\mathfrak q}}\le C^{1+m}\ \varepsilon^{-\frac m2}\ \sqrt{ m!}\ \Vert u\Vert_{L^{\mathfrak q}}.
\end{equation}
\end{proposition}

\begin{proof} Considering some $1\le\mathfrak q\le\infty$ fixed, let us begin by establishing the estimate \eqref{21102021E1} in terms of partial derivatives, instead of differentials. More precisely, let us prove that there exists a positive constant $C>0$ such that for all $\alpha\in\mathbb N^n$, $\varepsilon>0$ and $u\in\mathscr S(\mathbb R^n)$,
\begin{equation}\label{21102021E2}
	\big\Vert\langle Nx\rangle^{-\vert\alpha\vert_1}\partial^{\alpha}_x((e^{-\frac{\varepsilon}2\vert\xi - Nx\vert^2})^wu)\big\Vert_{L^{\mathfrak q}}\le C^{1+\vert\alpha\vert_1}\ \varepsilon^{-\frac{\vert\alpha\vert_1}2}\ \sqrt{\alpha!}\ \Vert u\Vert_{L^{\mathfrak q}},
\end{equation}
where $\vert\alpha\vert_1 = \alpha_1+\cdots+\alpha_n$, $\partial^{\alpha}_x = \partial^{\alpha_1}_{x_1}\cdots\partial^{\alpha_n}_{x_n}$, $x_1,\cdots,x_n$ denoting the coordinates in the canonical basis of $\mathbb R^n$ as usual, and $\alpha! = \alpha_1!\cdots\alpha_n!$. Let $\varepsilon>0$ and $u\in\mathscr S(\mathbb R^n)$ be fixed. In order to alleviate the writing, we set $f = (e^{-\frac{\varepsilon}2\vert\xi - Nx\vert^2})^wu$. Also fixing $x\in\mathbb R^n$, we first deduce from Corollary \ref{cor:Ng} that
$$f(x) = (2\pi\varepsilon)^{-n/2}\int_{\mathbb R^n}e^{-\frac1{2\varepsilon}\vert x-y\vert^2+i(x-y)\cdot Nx}u(y)\,\mathrm dy
= (2\pi\varepsilon)^{-n/2}\int_{\mathbb R^n}e^{-\frac1{2\varepsilon}\vert y\vert^2+iy\cdot Nx}u(x-y)\,\mathrm dy.$$
Since the function $u$ belongs to the Schwartz space $\mathscr S(\mathbb R^n)$ and that the integrand of the above integral is smooth, we deduce from Leibniz' formula and an integration by parts that
\begin{align*}
	\partial^{\alpha}_xf(x)
	& = (2\pi\varepsilon)^{-n/2}\int_{\mathbb R^n}\sum_{\alpha'\le\alpha}\binom{\alpha}{\alpha'}\partial^{\alpha'}_x(e^{-\frac1{2\varepsilon}\vert y\vert^2+iy\cdot Nx})\partial^{\alpha-\alpha'}_x(u(x-y))\,\mathrm dy \\[5pt]
	& = (2\pi\varepsilon)^{-n/2}\int_{\mathbb R^n}\sum_{\alpha'\le\alpha}\binom{\alpha}{\alpha'}(-iNy)^{\alpha'}e^{-\frac1{2\varepsilon}\vert y\vert^2+iy\cdot Nx}\partial^{\alpha-\alpha'}_y(-u(x-y))\,\mathrm dy \\[5pt]
	& = (2\pi\varepsilon)^{-n/2}\int_{\mathbb R^n}\sum_{\alpha'\le\alpha}\binom{\alpha}{\alpha'}\partial^{\alpha-\alpha'}_y\big((-iNy)^{\alpha'}e^{-\frac1{2\varepsilon}\vert y\vert^2+iy\cdot Nx}\big)u(x-y)\,\mathrm dy,
\end{align*}
where we used the fact that the matrix $N$ is skew-symmetric. Another use of Leibniz' formula implies that the above derivatives are given by
\begin{align*}
	\partial^{\alpha-\alpha'}_y\big((-iNy)^{\alpha'}(e^{-\frac1{2\varepsilon}\vert y\vert^2+iy\cdot Nx})\big) 
	& = \sum_{\alpha''\le\alpha-\alpha'}\binom{\alpha-\alpha'}{\alpha''}\partial^{\alpha-\alpha'-\alpha''}_y\big((-iNy)^{\alpha'}e^{-\frac1{2\varepsilon}\vert y\vert^2}\big)\partial^{\alpha''}_y(e^{iy\cdot Nx}) \\[5pt]
	& = \sum_{\alpha''\le\alpha-\alpha'}\binom{\alpha-\alpha'}{\alpha''}\partial^{\alpha-\alpha'-\alpha''}_y\big((-iNy)^{\alpha'}e^{-\frac1{2\varepsilon}\vert y\vert^2}\big)(iNx)^{\alpha''}(e^{iy\cdot Nx}).
\end{align*}
Gathering these two equalities, we therefore deduce that
\begin{multline*}
	\big\vert\langle Nx\rangle^{-\vert\alpha\vert_1}\partial^{\alpha}_xf(x)\big\vert
	\le (2\pi\varepsilon)^{-n/2}\sum_{\alpha'\le\alpha}\sum_{\alpha''\le\alpha-\alpha'}\binom{\alpha}{\alpha'}\binom{\alpha-\alpha'}{\alpha''} \\[5pt]
	\times\int_{\mathbb R^n}\big\vert\partial^{\alpha-\alpha'-\alpha''}_y\big((Ny)^{\alpha'}e^{-\frac1{2\varepsilon}\vert y\vert^2}\big)\big\vert\langle Nx\rangle^{-\vert\alpha\vert_1}\vert(Nx)^{\alpha''}\vert\vert u(x-y)\vert\, \mathrm dy.
\end{multline*}
Using the fact that $\vert\alpha''\vert_1\le\vert\alpha\vert_1$ and Young's convolution inequality, we obtain the following estimate
\begin{align}\label{21102021E3}
	&\ \big\Vert\langle Nx\rangle^{-\vert\alpha\vert_1}\partial^{\alpha}_xf\big\Vert_{L^{\mathfrak q}} \\[5pt]
	\le &\ (2\pi\varepsilon)^{-n/2}\sum_{\alpha'\le\alpha}\sum_{\alpha''\le\alpha-\alpha'}\binom{\alpha}{\alpha'}\binom{\alpha-\alpha'}{\alpha''}
	\big\Vert\big\vert\partial^{\alpha-\alpha'-\alpha''}_y\big((Ny)^{\alpha'}e^{-\frac1{2\varepsilon}\vert y\vert^2}\big)\big\vert\ast\vert u\vert\big\Vert_{L^{\mathfrak q}} \nonumber \\[5pt]
	\le &\ (2\pi\varepsilon)^{-n/2}\sum_{\alpha'\le\alpha}\sum_{\alpha''\le\alpha-\alpha'}\binom{\alpha}{\alpha'}\binom{\alpha-\alpha'}{\alpha''}
	\big\Vert\partial^{\alpha-\alpha'-\alpha''}_y\big((Ny)^{\alpha'}e^{-\frac1{2\varepsilon}\vert y\vert^2}\big)\big\Vert_{L^1}\Vert u\Vert_{L^{\mathfrak q}}. \nonumber
\end{align}
We now aim at controlling the above $L^1$ norm. First of all, by an homogeneity argument and a change of variable, we have
$$\big\Vert\partial^{\alpha-\alpha'-\alpha''}_y\big((Ny)^{\alpha'}e^{-\frac1{2\varepsilon}\vert y\vert^2}\big)\big\Vert_{L^1} = \frac{(2\varepsilon)^{\vert\alpha'\vert_1/2+n/2}}{(2\varepsilon)^{\vert\alpha-\alpha'-\alpha''\vert_1/2}}\big\Vert\partial^{\alpha-\alpha'-\alpha''}_y\big((Ny)^{\alpha'}e^{-\vert y\vert^2}\big)\big\Vert_{L^1}.$$
Then, we will use the fact that the standard Gaussian function enjoys Gelfand-Shilov regularity. More precisely, it follows from Example 6.3.1 and Theorem 6.1.6 in \cite{NR10} that there exists a positive constant $C_1>1$ such that for all $\beta,\gamma\in\mathbb N^n$,
$$\big\Vert y^{\beta}\partial^{\gamma}_y(e^{-\vert y\vert^2})\big\Vert_{L^{\infty}}\le C_1^{1+\vert\beta\vert_1+\vert\gamma\vert_1}\ \sqrt{\beta!}\ \sqrt{\gamma!}.$$
By using Leibniz's formula, we get that there exists another positive constant $C_2>0$ such that for all $m\geq0$ and $\beta,\gamma\in\mathbb N^n$,
$$\big\Vert\langle y\rangle^m\partial^{\beta}_y((Ny)^{\gamma}e^{-\vert y\vert^2})\big\Vert_{L^{\infty}}\le C_2^{1+m+\vert\beta\vert_1+\vert\gamma\vert_1}\ \sqrt{m!}\ \sqrt{\beta!}\ \sqrt{\gamma!}.$$
As a consequence, we obtain that for all $\alpha,\alpha',\alpha''\in\mathbb N^n$ satisfying $\alpha'\le\alpha$ and $\alpha''\le \alpha-\alpha'$,
\begin{align*}
	\big\Vert\partial^{\alpha-\alpha'-\alpha''}_y\big((Ny)^{\alpha'}e^{-\vert y\vert^2}\big)\big\Vert_{L^1}
	& = \big\Vert\langle y\rangle^{-n-1}\langle y\rangle^{n+1}\partial^{\alpha-\alpha'-\alpha''}_y\big((Ny)^{\alpha'}e^{-\vert y\vert^2}\big)\big\Vert_{L^1} \\[5pt]
	& \le \big\Vert\langle y\rangle^{-n-1}\big\Vert_{L^1}\big\Vert\langle y\rangle^{n+1}\partial^{\alpha-\alpha'-\alpha''}_y\big((Ny)^{\alpha'}e^{-\vert y\vert^2}\big)\big\Vert_{L^{\infty}} \\[5pt] 
	& \le \big\Vert\langle y\rangle^{-n-1}\big\Vert_{L^1}C_2^{2+n+\vert\alpha-\alpha'-\alpha''\vert_1}\ \sqrt{(n+1)!}\ \sqrt{(\alpha-\alpha'-\alpha'')!}\ \sqrt{\alpha'!} \\[5pt]
	& \le \big\Vert\langle y\rangle^{-n-1}\big\Vert_{L^1}C_2^{2+n+\vert\alpha-\alpha'-\alpha''\vert_1}\ \sqrt{(n+1)!}\ \sqrt{\alpha!}.
\end{align*}
Plugging this estimate in \eqref{21102021E3}, and using the fact that $\sum_{\alpha'\le\alpha}\binom{\alpha}{\alpha'} = 2^{\vert\alpha\vert_1},$
we deduce that the estimate \eqref{21102021E2} holds. Let us now derive the estimate \eqref{21102021E1} from \eqref{21102021E2}. First, notice that the norm we aim at bounding is given by
$$\big\Vert\langle Nx\rangle^{-m}\mathrm d^mf\big\Vert_{L^{\mathfrak q}}^{\mathfrak q} = \int_{\mathbb R^n}\Vert\langle Nx\rangle^{-m}\mathrm d^mf(x)\Vert^{\mathfrak q}_{\mathcal L^m}\,\mathrm dx,$$
where $\mathcal L^m$ denotes the space of continuous $m$-linear forms on $\mathbb R^n$. Recalling that for all $m\geq0$, $x\in\mathbb R^n$ and $h_1,\ldots,h_m\in\mathbb R^n$, we have
\begin{equation}\label{21102021E6}
	\mathrm d^mf(x)\cdot(h_1,\cdots,h_m) = \sum_{1\le i_1,\cdots,i_m\le n}\partial^m_{x_{i_1},\cdots,x_{i_m}}f(x)(h_1)_{i_1}\cdots(h_m)_{i_m},
\end{equation}
we get that
$$\Vert\mathrm d^mf(x)\Vert_{\mathcal L^m}\le\sum_{1\le i_1,\cdots,i_m\le n}\big\vert\partial^m_{x_{i_1},\cdots,x_{i_m}}f(x)\big\vert.$$
We therefore deduce from \eqref{21102021E2} that for all $m\geq0$,
$$
\big\Vert\langle Nx\rangle^{-m}\mathrm d^mf\big\Vert_{L^{\mathfrak q}}  \le \sum_{1\le i_1,\cdots,i_m\le n}\big\Vert\langle Nx\rangle^{-m}\partial^m_{x_{i_1},\cdots,x_{i_m}}f\big\Vert_{L^{\mathfrak q}} 
	 \le \frac{n^mC^{1+m}}{\varepsilon^{\frac m2}}\ \sqrt{m!}\ \Vert u\Vert_{L^{\mathfrak q}}.
$$

\end{proof}

By using Proposition \ref{prop:smoothtwist} and the key decomposition \eqref{eq:the_dec}, we can now tackle the proof of Theorem \ref{thm:maindir}, which is a consequence of the following

\begin{corollary}\label{cor:localsmooth} Let $q:\mathbb R^{2n}\rightarrow\mathbb C$ be a complex-valued quadratic form with a nonnegative real part. Let $S$ be its singular space (defined by \eqref{def:S}) and $k_0$ be its global index (defined by \eqref{def:k_0}). Assume that $S$ is included in the graph of a real $n \times n$ matrix $G$, i.e. $S\subset\{(x,Gx) \ | \ x\in \mathbb{R}^n \}$. Then, for all $1\le\mathfrak p\le\mathfrak q\le\infty$, there exist $C>1$ and $t_0>0$ such that for all $m\geq0$, $0<t<t_0$ and $u\in\mathscr S(\mathbb R^n)$,
\begin{equation}\label{21102021E4}
	\big\Vert(\langle Gx\rangle+\langle \transp Gx\rangle)^{-m}\mathrm d^m(e^{-tq^w}u)\big\Vert_{L^{\mathfrak q}}\le\frac{C^{1+m}}{t^{(k_0+\frac12)m+c_{\mathfrak p,\mathfrak q}}}\ \sqrt{m!}\ \Vert u\Vert_{L^{\mathfrak p}},
\end{equation}
where the positive constant $c_{\mathfrak p,\mathfrak q}>0$ is defined by \eqref{21102021E7}.
\end{corollary}

\begin{proof} Let $1\le\mathfrak p\le\mathfrak q\le\infty$ be fixed. As in the proof of Proposition \ref{prop:smoothtwist}, we will derive a version of the estimate \eqref{21102021E4} stated with partial derivatives, instead of differentials. However, in opposite of what we have done in the proof of Proposition \ref{prop:smoothtwist}, we will not work with the partial derivates associated with the canonical basis of $\mathbb R^n$, but with other ones well adapted in the present situation, defined as follows. Let $M = (G+\transp G)/2$ be the symmetric part of the matrix $G$. Since the matrix $M$ is real and symmetric, it is diagonalisable in an orthonormal basis. Let $y_1,\cdots,y_n$ be the coordinates in this new basis and $\partial_{y_1},\cdots,\partial_{y_n}$ be the associated partial derivatives. Using the same notations as in the proof of Proposition \ref{prop:smoothtwist}, we will establish that there exist some positive constants $C>1$ and $t_0>0$ such that for all $\alpha\in\mathbb N^n$, $0<t<t_0$ and $u\in\mathscr S(\mathbb R^n)$,
\begin{equation}\label{21102021E5}
	\big\Vert(\langle Gx\rangle+\langle \transp Gx\rangle)^{-\vert\alpha\vert_1}\partial^{\alpha}_y(e^{-tq^w}u)\big\Vert_{L^{\mathfrak q}}\le\frac{C^{1+\vert\alpha\vert_1}}{t^{(k_0+\frac12)\vert\alpha\vert_1+c_{\mathfrak p, \mathfrak q}}}\ \sqrt{\alpha!}\ \Vert u\Vert_{L^{\mathfrak p}},
\end{equation}
where the constant $c_{\mathfrak p,\mathfrak q}$ is the same as in \eqref{21102021E7}. Then, we can deduce the estimate \eqref{21102021E4} by using \eqref{21102021E6} anew. As explained before stating Corollary \ref{cor:localsmooth}, the key point is to use the decomposition given by Theorem \ref{thm:the_dec}, which allows to write the evolution operators $e^{-tq^w}$ in the following way for all $0<t<t_0$, with $t_0>0$,
$$e^{-tq^w} = c_t \, e^{\frac{i}2Gx\cdot x}(e^{-\gamma t^{\alpha}|\xi - Nx|^2})^we^{-tp_t^w}(e^{-\gamma t^{\alpha}|\xi - Nx|^2})^we^{itD_t\nabla\cdot\nabla}e^{tM_tx\cdot\nabla}e^{\frac i2(tW_t - G)x\cdot x},$$
where $\gamma>0$ is a positive constant, $D_t, M_t$ and $W_t$ are real $n\times n$ matrices, $c_t>0$ is a positive constant, all depending smoothly on $t\in(-t_0,t_0)$, $\alpha = 2k_0+1$ and $N = (G-\transp{G})/2$ denotes the skew-symmetric part of $G$. Notice that we can assume the matrices $D_t$ to be symmetric, since we get that 
$${D_t\nabla\cdot\nabla} =\frac12( D_t+\transp{D_t}) \nabla \cdot \nabla .$$

Let $0<t<t_0$ and $u\in\mathscr S(\mathbb R^n)$ be fixed. We consider the time-dependent function $v_t$ defined by
\begin{equation}\label{21102021E9}
	v_t = e^{-tp_t^w}(e^{-\gamma t^{\alpha}|\xi - Nx|^2})^we^{itD_t\nabla\cdot\nabla}e^{tM_tx\cdot\nabla}e^{\frac i2(tW_t - G)x\cdot x}u.
\end{equation}
Notice that from Theorem 4.2 in \cite{Hor95}, $v_t$ is also a Schwartz function, since the evolution operators generated by accretive quadratic operators map $\mathscr S(\mathbb R^n)$ into itself. Setting $\varepsilon_t = 2\gamma t^{\alpha}$, we first get from Leibniz' formula that for all $\alpha\in\mathbb N^n$,
$$\partial^{\alpha}_y(e^{-tq^w}u) = c_t\sum_{\alpha'\le\alpha}\binom{\alpha}{\alpha'}\partial^{\alpha-\alpha'}_y(e^{\frac i2Gx\cdot x})\partial^{\alpha'}_y((e^{-\frac{\varepsilon_t}2 |\xi - Nx|^2 })^wv_t).$$
We therefore obtain that for all $\alpha\in\mathbb N^n$,
\begin{multline}\label{17112021E1}
	\big\Vert(\langle Gx\rangle+\langle \transp Gx\rangle)^{-\vert\alpha\vert_1}\partial^{\alpha}_y(e^{-tq^w}u)\big\Vert_{L^{\mathfrak q}}\le c_t\sum_{\alpha'\le\alpha}\binom{\alpha}{\alpha'}
	\big\Vert(\langle Gx\rangle + \langle\transp Gx\rangle)^{-\vert\alpha-\alpha'\vert_1}\partial^{\alpha-\alpha'}_y(e^{\frac i2Gx\cdot x})\big\Vert_{L^{\infty}} \\[5pt]
	\times\big\Vert(\langle Gx\rangle + \langle\transp Gx\rangle)^{-\vert\alpha'\vert_1}\partial^{\alpha'}_y((e^{-\frac{\varepsilon_t}2\vert\xi - Nx\vert^2})^wv_t)\big\Vert_{L^{\mathfrak q}}.
\end{multline}
The purpose is now to estimate the two terms appearing in the above sum. On the one hand, we need to bound the following $L^{\infty}$ norms for all $\beta\in\mathbb N^n$:
$$\big\Vert(\langle Gx\rangle + \langle\transp Gx\rangle)^{-\vert\beta\vert_1}\partial^{\beta}_y(e^{\frac i2Gx\cdot x})\big\Vert_{L^{\infty}}.$$
First, notice that in this term, the matrix $G$ can be replaced by $M$ (recall that this is its symmetric part) and the weight $\langle Gx\rangle + \langle\transp Gx\rangle$ can be replaced by $\langle Mx\rangle$, since we have that for all $x\in\mathbb R^n$,
$$\langle Mx\rangle\lesssim\langle Gx\rangle + \langle\transp Gx\rangle\quad\text{and}\quad e^{\frac i2Gx\cdot x} = e^{\frac i2Mx\cdot x}.$$
Moreover, since $y_1,\cdots,y_n$ denote the coordinates in a basis that diagonalises the matrix $M$, we have
\begin{equation}\label{10112021E1}
	\big\Vert\langle Mx\rangle^{-\vert\beta\vert_1}\partial^{\beta}_y(e^{\frac i2Mx\cdot x})\big\Vert_{L^{\infty}}
\le\prod_{j=1}^n\big\Vert\langle2\lambda_jy\rangle^{-\beta_j}\partial^{\beta_j}_y(e^{i\lambda_j\vert y\vert^2})\big\Vert_{L^{\infty}},
\end{equation}
where $\lambda_1,\ldots,\lambda_n\in\mathbb R$ are the eigenvalues of the matrix $M/2$. In the following, we consider the determination $\sqrt\cdot$ of the square root on $\mathbb C\setminus\mathbb R_-$. Let $\beta\geq0$ be a integer and $\lambda\in\mathbb R$ be one of the $\lambda_j$. We need to introduce the Hermite polynomial, see e.g. \cite[eq 5.5.3]{Sze75},
$$H_{\beta}(y) = (-1)^{\beta}e^{-y^2}\frac{\mathrm d^{\beta}}{\mathrm dy^{\beta}}(e^{-y^2}).$$
Let us recall that the following formula holds, see e.g. \cite[eq 5.5.4]{Sze75},
$$H_{\beta}(y) = \sum_{k=1}^{\lfloor\beta/2\rfloor}(-1)^k\frac{\beta!}{k!(\beta-2k)!}(2y)^{\beta-2k}.$$
As a consequence, we obtain that
\begin{align*}
	\big\Vert\langle2\lambda y\rangle^{-\beta}\partial^{\beta}_y(e^{i\lambda\vert y\vert^2})\big\Vert_{L^{\infty}}
	& = \big\Vert\langle2\lambda y\rangle^{-\beta}\sqrt{-i\lambda}^{\beta}(-1)^{\beta}H_{\beta}(\sqrt{-i\lambda}\,y)\big\Vert_{L^{\infty}} \\
	& \le \vert\sqrt{-i\lambda}\vert^{\beta}\sum_{k=1}^{\lfloor\beta/2\rfloor}\frac{2^{\beta-2k}\beta!}{k!(\beta-2k)!}\big\Vert\langle2\lambda y\rangle^{-\beta}(\sqrt{-i\lambda}\,y)^{\beta-2k}\big\Vert_{L^{\infty}} \\
	& \le \vert\sqrt{-i\lambda}\vert^{\beta}\,2^{\beta}\sum_{k=1}^{\lfloor\beta/2\rfloor}\frac{(2k)!}{k!}\binom{\beta}{2k}  \le \vert\sqrt{-i\lambda}\vert^{\beta}\,4^{\beta}\sum_{k=1}^{\lfloor\beta/2\rfloor}2^kk! 
\end{align*}
Moreover, since $\beta^{\beta}\le e^{\beta}\beta!$, we get that for all $1\le k\le\lfloor\beta/2\rfloor$, $k!\le k^k\le (\beta/2)^{\beta/2}\le(e/2)^{\beta/2}\sqrt{\beta!},$
which implies that
$$\big\Vert\langle\lambda y\rangle^{-\beta}\partial^{\beta}_y(e^{i\lambda\vert y\vert^2})\big\Vert_{L^{\infty}}
\le\vert\sqrt{-i\lambda}\vert^{\beta}\,8^{\beta}\,\bigg(\frac e2\bigg)^{\beta/2}\,\Big\lfloor\frac{\beta}2\Big\rfloor\,\sqrt{\beta!}.$$
As a consequence of this estimate and \eqref{10112021E1}, we deduce that there exists a positive constant $C_1>0$ such that for all $\beta\in\mathbb N^n$,
\begin{equation}\label{29102021E1}
	\big\Vert(\langle Gx\rangle + \langle\transp Gx\rangle)^{-\vert\beta\vert_1}\partial^{\beta}_y(e^{\frac i2Gx\cdot x})\big\Vert_{L^{\infty}}\lesssim\big\Vert\langle Mx\rangle^{-\vert\beta\vert_1}\partial^{\beta}_y(e^{\frac i2Mx\cdot x})\big\Vert_{L^{\infty}}\le C_1^{1+\vert\beta\vert_1}\,\sqrt{\beta!}.
\end{equation}
On the other hand, since $N$ is the skew-symmetric part of $G$, we also have
$$\forall x\in\mathbb R^n,\quad\langle Nx\rangle\lesssim\langle Gx\rangle + \langle\transp Gx\rangle,$$
and we deduce from Proposition \ref{prop:smoothtwist} that there exists another positive constant $C_2>0$ such that for all $\beta\in\mathbb N^n$ and $0<t<t_0$,
\begin{align}\label{29102021E2}
	\big\Vert(\langle Gx\rangle + \langle\transp Gx\rangle)^{-\vert\beta\vert_1}\partial^{\beta}_y((e^{-\frac{\varepsilon_t}2\vert\xi - Nx\vert^2})^wv_t)\big\Vert_{L^{\mathfrak q}}
	& \lesssim \big\Vert\langle Nx\rangle^{-\vert\beta\vert_1}\partial^{\beta}_y((e^{-\frac{\varepsilon_t}2\vert\xi - Nx\vert^2})^wv_t)\big\Vert_{L^{\mathfrak q}} \\[5pt]
	& \le \big\Vert\langle Nx\rangle^{-\vert\beta\vert_1}\mathrm d^{\beta}((e^{-\frac{\varepsilon_t}2\vert\xi - Nx\vert^2})^wv_t)\big\Vert_{L^{\mathfrak q}} \nonumber \\[5pt]
	& \le C_2^{1+\vert\beta\vert_1} \varepsilon_t^{-\frac{\vert\beta\vert}2} \, \sqrt{\beta!}\ \Vert v_t\Vert_{L^{\mathfrak q}}. \nonumber
\end{align}
Combining the estimates \eqref{29102021E1} and \eqref{29102021E2} with \eqref{17112021E1}, we obtain that
\begin{equation}\label{21102021E11}
	\big\Vert(\langle Gx\rangle + \langle\transp Gx\rangle)^{-\vert\alpha\vert_1}\partial^{\alpha}_x(e^{-tq^w}u)\big\Vert_{L^{\mathfrak q}}\le C_3^{1+\vert\alpha\vert_1} \varepsilon_t^{-\frac{\vert\alpha\vert_1}2} \ \sqrt{\alpha!}\ \Vert v_t\Vert_{L^{\mathfrak q}},
\end{equation}
where $C_3>0$ is a positive constant not depending on $\alpha\in\mathbb N^n$, $0<t<t_0$, neither $u\in\mathscr S(\mathbb R^n)$. 

Now, it only remains to estimate the $L^{\mathfrak q}$ norm of the function $v_t$. To that end, we need to study the action on the Lebesgue spaces of the operators involved in the definition \eqref{21102021E9} of the function $v_t$. First of all, notice that the operator $e^{(i/2)(tW_t-G)x \cdot x}$ is an isometry on $L^{\mathfrak p}$ so it plays no role, and $e^{tM_tx\cdot\nabla}$ is the flow of a transport equation, so it is invertible on $L^{\mathfrak p}$ and is close to the identity, and plays no role neither. Moreover, we get from Corollary \ref{cor:twisted_vs_schro} that for all $w\in\mathscr S(\mathbb R^n)$,
$$\big\| (e^{-\frac{\varepsilon}2 |\xi - Nx|^2 })^we^{itD_t\nabla\cdot\nabla}w\big\|_{L^\mathfrak q} \le C_4\,\varepsilon^{-\frac n{2\mathfrak r}}  \, (\mathrm{det} \, (\varepsilon^2 I_n + D^2 ))^{\frac1{2\mathfrak r} - \frac14}\ \| w\|_{L^\mathfrak p},$$
where $C_4>0$ is a positive constant only depending on the dimension, and
$\mathfrak r = (1- \mathfrak p^{-1} + \mathfrak q^{-1} )^{-1} .$
At last, it follows from Theorem \ref{thm:fifou} that for all $w\in\mathscr S(\mathbb R^n)$, $\Vert e^{-tp_t^w}w\Vert_{L^\mathfrak q}\leq\Vert w\Vert_{L^\mathfrak q}$. In a nutshell, we obtain that the $L^{\mathfrak q}$ norm of the function $v_t$ is bounded as follows
\begin{equation}\label{21102021E12}
	\Vert v_t\Vert_{L^{\mathfrak q}}\le C_5\,\varepsilon^{-\frac n{2\mathfrak r}}\, (\mathrm{det}(\varepsilon^2 I_n + t^2D_t^2 ))^{\frac1{2\mathfrak r} - \frac14}\,\Vert u\Vert_{L^\mathfrak p},
\end{equation}
with $C_5>0$ another positive constant only depending on the dimension. Recalling that $\varepsilon_t = 2\gamma t^{2k_0+1}$, we lastly need to control the right-hand side of the above estimate. We denote by $\lambda_{1,t},\ldots,\lambda_{n,t}$ the eigenvalues of the matrix $\varepsilon^2_tI_n + t^2D_t^2$. First, since the symmetric matrix $D^2_t$ is nonnegative, we get that
$$\forall j\in\{1,\ldots,n\},\quad\lambda_{j,t}\geq\varepsilon^2_t = 4\gamma^2t^{2(2k_0+1)}.$$
Let us denote by $\rho_t$ the spectral radius of the matrix $\varepsilon^2_tI_n + t^2D_t^2$. Since the norm $\Vert\cdot\Vert$ is induced by the canonical Euclidean norm on $\mathbb R^n$, we get that $\rho_t\le\Vert\varepsilon^2_tI_n + t^2D_t^2\Vert$. Moreover, the matrix $D_t^2$ is smooth with respect to the time-variable $t\in(-t_0,t_0)$, so that there exists a positive constant $C_6>0$ such that for all $1\le j\le n$,
$$\lambda_{j,t}\le\rho_t\le\Vert\varepsilon^2_tI_n + t^2D_t^2\Vert = t^2\Vert 4\gamma^2t^{4k_0}I_n + D_t^2\Vert\le C_6\,t^2.$$
The determinant of a matrix being equal to the product of its eigenvalues, we obtain that
\begin{equation}\label{21102021E10}
	(4\gamma^2)^nt^{2n(2k_0+1)}\le\mathrm{det}(\varepsilon^2_tI_n + t^2D_t^2)\le C_6^nt^{2n}.
\end{equation}
By gathering the estimates \eqref{21102021E11}, \eqref{21102021E12} and \eqref{21102021E10}, we deduce that there exists a positive constant $C>0$ such that for all $\alpha\in\mathbb N^n$, $0<t<t_0$ and $u\in\mathscr S(\mathbb R^n)$,
$$\big\Vert(\langle Gx\rangle+\langle \transp Gx\rangle)^{-\vert\alpha\vert_1}\partial^{\alpha}_x(e^{-tq^w}u)\big\Vert_{L^{\mathfrak q}}\le\frac{C^{1+\vert\alpha\vert_1}}{t^{(k_0+\frac12)\vert\alpha\vert_1+c_{\mathfrak p, \mathfrak q}}}\ \sqrt{\alpha!}\ \Vert u\Vert_{L^{\mathfrak p}},$$
where the positive constant $c_{\mathfrak p, \mathfrak q}>0$ is given by
$$\begin{array}{ll}
	\displaystyle c_{\mathfrak p, \mathfrak q} = \frac n{2\mathfrak r}(2k_0+1) - 2n\bigg(\frac1{2\mathfrak r}-\frac14\bigg) = \frac n{2\mathfrak r}(2k_0 + \mathfrak r - 1), & \text{when $1\le\mathfrak r\le 2$,} \\[15pt]
	\displaystyle c_{\mathfrak p, \mathfrak q} = \frac n{2\mathfrak r}(2k_0+1) - 2n(2k_0+1)\bigg(\frac1{2\mathfrak r}-\frac14\bigg) = \frac n{2\mathfrak r}(2k_0+1)(\mathfrak r-1), & \text{when $\mathfrak r>2$}.
\end{array}$$
\end{proof}

\section{Reciprocals}\label{sec:reciprocals}

The aim of this before last section is to prove Theorem \ref{thm:recip}, which is the reciprocal of Theorem \ref{thm:maindir} proven in Section \ref{sec:localsmooth}. The strategy is to use the polar decomposition \eqref{eq:ours} introduced by the authors in \cite{AB21}. Thanks to this decomposition, it is sufficient to consider only semigroups generated by quadratic selfadjoint operators.

The following lemma will be key in this section.

\begin{lemma}\label{02092021L1} Let $q:\mathbb R^{2n}\rightarrow\mathbb R_+$ be a nonnegative quadratic form which does not depend on the variable $\xi_n$. Then, we have that for all $f\in\mathscr S(\mathbb R^{n-1})$ and $g\in\mathscr S(\mathbb R)$,
$$e^{-q^w}(f\otimes g)(x,x_n) = (e^{-q_{x_n}^w}f)(x)g(x_n),\quad (x,x_n)\in\mathbb R^{n-1}\times\mathbb R,$$
where the symbol $q_{x_n}:\mathbb R^{2(n-1)}\rightarrow\mathbb R$ is defined by
$$q_{x_n}(x,\xi) = q(x,x_n,\xi,0),\quad (x,\xi)\in\mathbb R^{2(n-1)}.$$
\end{lemma}

\begin{proof} Being given $u\in\mathscr S(\mathbb R^n)$ a Schwartz function, we recall from Theorem 4.2 in \cite{Hor95} that $e^{-q^w}u$ is also a Schwartz function. Then, as a consequence\footnote{replacing $\{1\}$ by $\{x_n\}$ in its proof.} of \cite[Lemma 4]{Ber21}, we have\footnote{a priori in $L^2(\mathbb{R}^{n-1})$ but actually in $\mathscr{S}(\mathbb{R}^{n})$ since  $e^{-q^w}u$ is a Schwartz function.}
$$\forall x_n\in\mathbb R, \quad(e^{-q^w}u)_{\vert\mathbb R^{n-1}\times\{x_n\}} = (e^{-q_{x_n}^w}u_{\vert\mathbb R^{n-1}\times\{x_n\}}).$$
Applying this identity with $u = f \otimes g$, where $f\in\mathscr S(\mathbb R^{n-1})$ and $g\in\mathscr S(\mathbb R)$, it comes naturally, as expected, that
$$\forall x_n\in\mathbb R, \quad e^{-q^w}(f\otimes g)(\cdot,x_n) = e^{-q_{x_n}^w}(g(x_n)f) = (e^{-q_{x_n}^w}f)g(x_n).$$
\end{proof}

As announced, we begin by studying the reciprocal of Theorem \ref{thm:maindir} for semigroups generated by nonnegative quadratic forms.

\begin{proposition}\label{01102021P1} Let $q:\mathbb R^{2n}\rightarrow\mathbb R_+$ be a nonnegative quadratic form. Let $S\subset\mathbb R^{2n}$ be the isotropic cone of $q$. Assume that the geometric condition $S\cap(\{0\}\times\mathbb R^n)\ne\{0\}$ holds. Then: 
\begin{enumerate}
	\item[$(i)$] There exists $u\in L^2(\mathbb R^n)$ such that $e^{-q^w}u$ is not a continuous function.
	\item[$(ii)$] For all $2<\mathfrak p\le\infty$, there exists $u_{\mathfrak p}\in L^2(\mathbb R^n)$ such that $e^{-q^w}u_{\mathfrak p}\notin L^{\mathfrak p}(\mathbb R^n)$.
\end{enumerate}
\end{proposition}

\begin{proof} Let us assume once and for all that $S\cap(\{0\}\times\mathbb R^n)\ne\{0\}$. Therefore, there exists a vector $\xi_0\in\mathbb R^n \setminus \{0\}$ such that $(0,\xi_0)\in S$. We begin by checking that we can choose $\xi_0 = e_n := (0,\cdots,0,1)$. Let us consider an invertible matrix $A\in\mathrm{GL}_n(\mathbb R)$ such that $A^Te_n = \xi_0$.  The change of variable $x \mapsto Ax$ in the equation $\partial_t u + q^w u =0$ writes\footnote{note that this formula also follows directly from the metaplectic invariance of the Weyl calculus.} as 
$$K_A^{-1}e^{-q^w}K_A = e^{-(q\circ T_A)^w}\quad\text{where}\quad T_A = \begin{pmatrix}
	A^{-1} & 0 \\
	0 & \transp A
\end{pmatrix}\in\mathrm{Sp}_{2n}(\mathbb R),$$
where $K_A$ is the unitary operator on $L^2(\mathbb R^n)$ defined by 
$$
K_Au = \sqrt{\vert\det A\vert}\, u(A\,\cdot),\quad u\in L^2(\mathbb R^n).
$$
Notice that $K_A$ is also a similarity on $L^{\mathfrak p}$ for all $2<\mathfrak p<+\infty$, and maps the space $C^0(\mathbb R^n)$ into itself. Moreover, since $(0,\xi_0)\in S$ and that the vector space $S$ is the isotropic cone of the quadratic form $q$, we have
$$(q\circ T_A)(0,e_n) = q(0, \transp Ae_n) = q(0,\xi_0) = 0.$$
We can therefore choose $\xi_0 = e_n$.

Let us now tackle the proof of the assertion $(i)$. We have to prove that there exists a function $u\in L^2(\mathbb R^n)$ such that $e^{-q^w}u\notin C^0(\mathbb R^n)$. To that end, for all $x_n\in\mathbb R$, we consider the degree-2 polynomial $q_{x_n}:\mathbb R^{2(n-1)}\rightarrow\mathbb R$ defined by
\begin{equation}\label{09092021E1}
	q_{x_n}(x,\xi) = q(x,x_n,\xi,0),\quad (x,\xi)\in\mathbb R^{2(n-1)}.
\end{equation}
We know from \cite[Corollary 7.9]{AB21} that the operator $e^{-q_0^w}$ generated by the nonnegative quadratic form $q_0$ is one-to-one, so there exists a function $f$ from the Schwartz space $\mathscr S(\mathbb R^{n-1})$ such that the function $e^{-q_0^w}f$ is not identically equal to zero. As a consequence,
\begin{equation}\label{02092021E1}
	\exists x_*\in\mathbb R^{n-1}, \quad (e^{-q_0^w}f)(x_*)\ne0.
\end{equation}
We set $g(x_n) = (\log(x_n))^{-2} \mathbbm{1}_{|x_n|<1}$ and $u = f\otimes g\in L^2(\mathbb R^n)$. Let us assume that the function $e^{-q^w}u$ is continuous at $(x_*,0)$. Since $q(0,e_n) = 0$, the quadratic form $q$ does not dependent on the variable $\xi_n\in\mathbb R^n$, and it follows from Lemma \ref{02092021L1} that $e^{-q^w}u$ is given by
\begin{equation}\label{09112021E1}
	(e^{-q^w}u)(x,x_n) = (e^{-q_{x_n}^w}f)(x)g(x_n),\quad (x,x_n)\in\mathbb R^{n-1}\times\mathbb R,
\end{equation}
where the symbol $q_{x_n}$ is the one defined in \eqref{09092021E1}. Notice that the function $e^{-q_{x_n}^w}f$ does not vanish at $(x_*,0)$ by definition of the point $x_*$:
$$(e^{-q_{x_n}^w}f)(x_*,0) = (e^{-q_0^w}f)(x_*)\ne0.$$
Moreover, the function $e^{-q_{x_n}^w}f$ is continuous, since we deduce from Lemma \ref{02092021L1} anew that
$$e^{-q_{x_n}^w}f = e^{-q^w}(f\otimes e^{-\vert x_n\vert^2})e^{\vert x_n\vert^2},$$
and $e^{-q^w}:\mathscr S(\mathbb R^n)\rightarrow\mathscr S(\mathbb R^n)$ from \cite[Theorem 4.2]{Hor95}. According to \eqref{09112021E1}, the function $g$ would therefore be continuous, but is not.

Now, we prove the assertion $(ii)$. Given $2<\mathfrak p\le\infty$, we aim at finding a function $u_{\mathfrak p}\in L^2(\mathbb R^n)$ such that $e^{-q^w}u_{\mathfrak p}\notin L^{\mathfrak p}(\mathbb R^n)$. We first assume that $2<\mathfrak p<\infty$. Let $f\in\mathscr S(\mathbb R^{n-1})$ be the same Schwartz function as in the previous paragraph. In particular, \eqref{02092021E1} holds, and since the function $e^{-q_0^w}f$ is continuous at $(x_*,0)$, where $x_*\in\mathbb R^{n-1}$ is the point appearing in \eqref{02092021E1}, we deduce that there exist $m>0$ and $r_*>0$ such that for all $(x,x_n)\in\mathbb R^{n-1}\times\mathbb R$,
\begin{equation}\label{01102021E1}
	\vert x-x_*\vert+\vert x_n\vert\le r_*\Rightarrow\big\vert(e^{-q_{x_n}^w}f)(x)\big\vert\geq m.
\end{equation}
Let us now consider the function $g_{\mathfrak p}\in L^2(\mathbb R)\setminus L^{\mathfrak p}(\mathbb R)$ defined by
$$g_{\mathfrak p}(x_n) = \frac1{\vert x_n\vert^{1/\mathfrak p}}\mathbbm1_{\vert x_n\vert\le r_*},\quad x_n\in\mathbb R.$$
Setting $u_{\mathfrak p} = f\otimes g_{\mathfrak p}\in L^2(\mathbb R^n)$, it follows from \eqref{09112021E1} and \eqref{01102021E1} that
$$\big\Vert e^{-q^w}u_{\mathfrak p}\big\Vert^{\mathfrak p}_{L^{\mathfrak p}}=\int_{\mathbb R\times\mathbb R^{n-1}}\big\vert(e^{-q_{x_n}^w}f)(x)g_{\mathfrak p}(x_n)\big\vert^{\mathfrak p}\,\mathrm dx\mathrm dx_n\geq m^{\mathfrak p}\int_{\vert x-x_*\vert+\vert x_n\vert\le r_*}\vert g_{\mathfrak p}(x_n)\vert^{\mathfrak p}\,\mathrm dx\mathrm dx_n = +\infty.$$
As a consequence, $e^{-q^w}u_{\mathfrak p}\notin L^{\mathfrak p}(\mathbb R^n)$ as expected. The case $\mathfrak p=\infty$ can be treated the same way by considering $g_{\infty}\in L^2(\mathbb R)\setminus L^{\infty}(\mathbb R)$.
\end{proof}

We can now tackle the proof of Theorem \ref{thm:recip}.

\begin{proof}[Proof of Theorem \ref{thm:recip}] This theorem is a quite  straightforward consequence of Proposition \ref{01102021P1} and the polar decomposition introduced in \cite{AB21}. Indeed, on the one hand, Theorem 2.1 in \cite{AB21} implies that there exist a nonnegative quadratic form $a:\mathbb R^{2n}\rightarrow\mathbb R_+$ and a unitary operator $U$ on $L^2(\mathbb R^n)$ such that 
\begin{equation}\label{09112021E2}
	e^{-q^w} = e^{-a^w} U.
\end{equation}
On the other hand, assuming that $S \cap(\mathbb R^n\times\{ 0 \})^\perp \ne \{0\}$, and since $S$ is the isotropic cone of the quadratic form $a$, see Proposition \ref{prop:singspacepolar} in appendix, we deduce that there exists a function $u\in L^2(\mathbb R^n)$ such that $e^{-a^w}u$ is not continuous, and for all $2<\mathfrak p\le\infty$, there exists $u_{\mathfrak p}\in L^2(\mathbb R^n)$ such that $e^{-a^w}u_{\mathfrak p}\notin L^{\mathfrak p}(\mathbb R^n)$. Setting $\tilde u = U^{-1}u\in L^2(\mathbb R^n)$ and $\tilde u_{\mathfrak p} = U^{-1}u_{\mathfrak p}\in L^2(\mathbb R^n)$, we therefore get from \eqref{09112021E2} that $e^{-q^w}\tilde u\notin C^0(\mathbb R^n)$ and $e^{-q^w}\tilde u_{\mathfrak p}\notin L^{\mathfrak p}(\mathbb R^n)$ for all $2<\mathfrak p\le\infty$. This proves the result by contraposition.
\end{proof}

\section{Appendix} 

\subsection{The geometric condition}

\begin{lemma} \label{lem:graph} Let $n\geq 1$ and $S$ be a subspace of $\mathbb R^{2n}$. The following assertions are equivalent:
\begin{enumerate}[parsep=2pt,itemsep=0pt,topsep=2pt]
\item[$(i)$] $S$ is included in the graph of a $n\times n$ real matrix $G$, i.e. $S \subset \{ (x,Gx) \ | \ x\in \mathbb R^n \}$,
\item[$(ii)$] $S \cap  (\mathbb R^n  \times \{ 0 \})^\perp = \{0\}$.
\end{enumerate}
\end{lemma}

\begin{proof} Since $(\mathbb{R}^n  \times \{ 0 \})^\perp =   \{ 0 \} \times \mathbb{R}^n$ and $G0=0$, it is obvious that $(i)$ implies $(ii)$. So from now, we only aim at proving that $(ii)$ implies $(i)$. We assume that $S \cap  (\mathbb{R}^n  \times \{ 0 \})^\perp = \{0\}$ and we denote by $\Pi$ the orthogonal projection on $\mathbb{R}^n  \times \{ 0 \}$. We set $V = \Pi S$. First, we check that for all $x\in V$ there exists a unique $g(x) \in \mathbb{R}^n$ such that $(x,g(x)) \in S$. Indeed, the existence is obvious by definition of $\Pi$. Moreover if both $(x,\xi) \in S$ and $(x,\zeta) \in S$ then, since $S$ is a vector space, $(0,\xi-\zeta) \in S$. But, by assumption $S \cap  (  \{ 0 \} \times \mathbb{R}^n) = \{0\}$ and so $(0,\xi-\zeta)= 0$, i.e. $\xi=\zeta$. Finally, we just have to check that $x\mapsto g(x)$ is linear but this is a straightforward consequence of the fact that $S$ is a vector space.
\end{proof}

\subsection{Singular space and polar decomposition} Let $q:\mathbb R^{2n}\rightarrow\mathbb C$ be a complex-valued quadratic form with a nonnegative real part. We have proven in \cite[Theorem 2.1]{AB21} that there exists a family $(a_t)_{t\in\mathbb R}$ of nonnegative quadratic forms $a_t:\mathbb R^{2n}\rightarrow\mathbb R_+$ depending analytically on the time-variable $t\in\mathbb R$ and a family $(U_t)_{t\in\mathbb R}$ of unitary operators on $L^2(\mathbb R^n)$ such that
$$\forall t\geq0,\quad e^{-tq^w} = e^{-ta_t^w}U_t.$$
In the following proposition, we make the link between the isotropic cones of the quadratic forms $a_t$ and the singular space $S$ of $q$, defined by \eqref{def:S}.

\begin{proposition}\label{prop:singspacepolar} The singular space $S$ of the quadratic form $q$ is the isotropic cone of the quadratic form $a_t$ for all $t>0$.
\end{proposition}

\begin{proof} First of all, the fact that the isotropic cone of $a_t$ is contained in the singular space $S$ is a straightforward consequence of the definition \eqref{def:S} of the singular space $S$, the definition \eqref{def:k_0} of the global index $0\le k_0\le 2n-1$, the following estimate stated in \cite[Theorem 2.2]{AB21}
$$a_t(X)\geq c\sum_{j=0}^{k_0} t^{2j} \Reelle q\big((\Imag F)^j X\big),\quad t\in[0,T],\, X\in\mathbb R^{2n},$$
with $T\in(0,1)$, and an analyticity argument (recall that $a_t$ depends analytically on the time variable $t\in\mathbb R$).

For the reciprocal, a simple way to check that $a_t$ vanishes on $S$ is to use the Baker-Campbell-Hausdorff formula thanks to the following relation, see \cite[Theorem 3.2]{AB21},
$$e^{-4itJA_t} = e^{-2itJQ} e^{-2itJ\overline Q},$$
where $A_t$ (resp. $Q$) denotes the matrix of $a_t$ (resp. $q$). Indeed, it proves that $tA_t$ is an analytic function of $(t\Reelle JQ,t\Imag JQ)$. But since $tA_t$ is a real matrix, each term of its analytic expansion has to contain at least one factor of the form $t\Reelle JQ$. As a consequence, each term is of the form $M_t(\Reelle JQ)(t\Imag JQ)^k$ for some matrix $M_t$ and some integer $k\in\mathbb N$. Therefore, by definition of $S$, each term of this analytic expansion vanishes on $S$, and so $A_t$ also vanishes on $S$.
\end{proof}


\subsection*{Conflict of interest}

The authors declare that they have no conflict of interest.


\begin{thebibliography}{10000000}

\bibitem[Alp20]{Alp20b} 
{\sc P. Alphonse}, 
\emph{Null-controllability of evolution equations associated with fractional Shubin operators through quantitative Agmon estimates}, to appear in Annales de l'Institut Fourier (2022), \href{https://arxiv.org/pdf/2012.04374.pdf}{arXiv:2012.04374}.

\bibitem[Alp21]{Alp20a} 
{\sc P. Alphonse}, 
\emph{Quadratic differential equations: partial Gelfand-Shilov smoothing effect and null-controllability}, \href{https://doi.org/10.1017/S1474748019000628}{J. Inst. Math. Jussieu} 20 (2021), no. 6, pp. 1749-1801.

\bibitem[AB20]{AB20} 
{\sc P. Alphonse, J. Bernier}, 
\emph{Smoothing properties of fractional Ornstein-Uhlenbeck semigroups and null-controllability}, 
\href{https://doi.org/10.1016/j.bulsci.2020.102914}{Bull. Sci. Math.} 165 (2020), 102914, 52 pp.

\bibitem[AB21]{AB21} 
{\sc P. Alphonse, J. Bernier}, 
\emph{Polar decomposition of semigroups generated by non-selfadjoint quadratic differential operators and regularizing effects}, to appear in Ann. Scient. Ec. Norm. Sup (2021), \href{https://arxiv.org/abs/1909.03662}{arXiv:1909.03662}.

\bibitem[Ber21]{Ber21} 
{\sc J. Bernier}, 
\emph{Exact splitting methods for semigroups generated by inhomogeneous quadratic differential operators}, 
\href{https://doi.org/10.1007/s10208-020-09487-4}{Found. Comput. Math.} 21 (2021), no. 5, pp. 1401-1439.

\bibitem[BCOR09]{BCOR09} 
{\sc S. Blanes, F. Casas, J.A. Oteo, J. Ros}, 
\emph{The Magnus expansion and some of its applications}, 
\href{https://doi.org/10.1016/j.physrep.2008.11.001}{Phys. Rep.} 470 (2009), no. 5-6, pp. 151-238.

\bibitem[Fol89]{Fol89} 
{\sc G.B. Folland}, 
\emph{Harmonic Analysis in Phase Space}, 
(AM-122). Princeton University Press (1989).




\bibitem[HLW06]{HLW06}
{\sc E. Hairer, C. Lubich, G. Wanner}, 
\emph{Geometric numerical integration: Structure-Preserving Algorithms for Ordinary Differential Equations}, 
\href{https://doi.org/10.1007/3-540-30666-8}{Springer Series in Computational Mathematics} (2006). 


\bibitem[Her07]{Her07} 
{\sc F. Hérau},
\emph{Short and long time behavior of the Fokker-Planck equation in a confining potential and applications},
\href{https://doi.org/10.1016/j.jfa.2006.11.013}{J. Funct. Anal.} 244 (2007), no. 1, pp. 95-118.

\bibitem[HP09]{HP09} 
{\sc M. Hitrik, K. Pravda-Starov}, \emph{Spectra and semigroup smoothing for non-elliptic quadratic operators}, \href{https://doi.org/10.1007/s00208-008-0328-y}{Math. Ann.}  344 (2009), no. 4, pp. 801-846.

\bibitem[HPV17]{HPV17} 
{\sc M. Hitrik, K. Pravda-Starov, J. Viola}, \emph{Short-time asymptotics of the regularizing effect for semigroups generated by quadratic operators}, \href{https://doi.org/10.1016/j.bulsci.2017.07.003}{Bull. Sci. Math.} 141 (2017), no. 7, pp. 615-675.

\bibitem[HPV18]{HPV18} 
{\sc M. Hitrik, K. Pravda-Starov, J. Viola}, \emph{From semigroups to subelliptic estimates for quadratic operators}, \href{https://doi.org/10.1090/tran/7251}{Trans. Amer. Math. Soc.} 370 (2018), pp. 7391-7415.


\bibitem[Hor85]{Hor85}
{\sc L. H\"ormander}, 
\emph{The analysis of linear partial differential operators III}, 
\href{https://doi.org/10.1007/978-3-540-49938-1}{Springer Verlag} (1985).


\bibitem[Hor95]{Hor95}
{\sc L. H\"ormander}, 
\emph{Symplectic classification of quadratic forms, and general Mehler formulas}, 
\href{https://doi.org/10.1007/BF02572374}{Math. Z.} 219 (1995), no. 3, pp. 413-449.

\bibitem[HJ90]{HJ90} 
R. A. \textsc{Horn}, C. R. \textsc{Johnson}, 
\textit{Matrix analysis}, 
\href{https://www.cambridge.org/fr/academic/subjects/mathematics/algebra/matrix-analysis-2nd-edition?format=PB&isbn=9780521548236}{Cambridge University Press}, Cambridge (1990).

\bibitem[Kol34]{Kol34} 
{\sc A. Kolmogoroff},
Zuf\"allige {B}ewegungen (zur {T}heorie der {B}rownschen {B}ewegung),
Ann. of Math. (2) 35 (1934), no. 1, pp. 116-117.

 \bibitem[Lieb90]{Lieb90}
{\sc E.H. Lieb}, 
\emph{Gaussian kernels have only Gaussian maximizers}, 
\href{https://doi.org/10.1007/BF01233426}{Invent. Math.} 102 (1990), pp. 179-208.

 \bibitem[Neg95]{Neg95}
{\sc E.R. Negrin}, 
\emph{Operators with Complex Gaussian Kernels: Boundedness Properties}, 
\href{https://doi.org/10.1007/BF01233426}{Proc. Amer. Math. Soc.} 123 (1995), no. 4, pp. 1185-1190.

\bibitem[NR10]{NR10} 
F. \textsc{Nicola}, L. \textsc{Rodino}, 
\textit{Global pseudo-differential calculus on Euclidean spaces}, 
\href{https://www.springer.com/gp/book/9783764385118}{Pseudo-Differential Operators, Theory and Applications}, Vol. 4, Birkh\"auser Verlag, Basel (2010).

\bibitem[PRW18]{PRW18}
K. \textsc{Pravda-Starov}, L. \textsc{Rodino}, P. \textsc{Wahlberg},
\textit{Propagation of Gabor singularities for Schr\"odinger equations with quadratic Hamiltonians},
\href{https://onlinelibrary.wiley.com/doi/10.1002/mana.201600410}{Math. Nachr.} 291 (2018), no. 1, pp. 128-159.


\bibitem[Sze75]{Sze75}
G. \textsc{Szeg\H o},
\textit{Orthogonal polynomials},
\href{https://www.ams.org/books/coll/023/}{American Mathematical Society Colloquium Publications}, Vol. XXIII, American Mathematical Society, Providence, R.I. (1975).

\bibitem[Vio17]{Vio17}
{\sc J. Viola}, 
\textit{The elliptic evolution of non-self-adjoint degree-2 Hamiltonians}, 
preprint (2017), \href{https://arxiv.org/abs/1701.00801}{arxiv 1701.00801}.

\bibitem[Wah18]{Wah18}
P. \textsc{Wahlberg},
\textit{The {Gabor} wave front set of compactly supported distributions},
\href{https://link.springer.com/book/10.1007/978-3-030-36138-9}{Advances in microlocal and time-frequency analysis}, Contributions of the conference on microlocal and time-frequency analysis,
Cham: Birkh{\"a}user (2018).

\bibitem[Wei79]{Wei79}
{\sc F.B. Weissler}, 
\emph{Two-point inequalities, the Hermite semigroup and the Gauss-Weierstrass semigroup}, 
\href{https://doi.org/10.1016/0022-1236(79)90080-6}{J. Funct. Anal.} 32 (1979), no. 1, pp. 102-121.

\bibitem[Whi21a]{Whi21a}
{\sc F. White}, 
\textit{Propagation of Global Analytic Singularities for Schr\"odinger Equations with Quadratic Hamiltonians}, \href{https://www.sciencedirect.com/science/article/pii/S0022123622001896}{Journal of Functional Analysis} 283 (2022), no. 3, 45 pp.
 
\bibitem[Whi21b]{Whi21b}
{\sc F. White}, 
\textit{$L^p$-bounds for semigroups generated by non-elliptic quadratic differential operators}, to appear in Journal of Spectral Theory (2021),
\href{https://arxiv.org/abs/2104.14613}{arXiv:2104.14613}.
 
\end{thebibliography}
\end{document}